\documentclass[11pt]{article}
\usepackage[english]{babel}
\usepackage[letterpaper,top=2cm,bottom=2cm,left=3cm,right=3cm,marginparwidth=1.75cm]{geometry}
\usepackage[autostyle]{csquotes}

\usepackage{amsmath, amsfonts, amsthm}
\usepackage{amsmath, amsfonts, amsthm}
\usepackage{graphicx}
\usepackage[colorlinks=true, allcolors=blue]{hyperref}
\usepackage{booktabs}   
\usepackage{multirow}  
\usepackage{siunitx} 
\usepackage{mathtools} 
\allowdisplaybreaks 
\usepackage[capitalize]{cleveref}
\usepackage{bbm}
\usepackage{nicematrix}
\usepackage{tikz}
\usetikzlibrary{cd}
\usepackage[ruled,linesnumbered]{algorithm2e}

\newtheorem{theorem}{Theorem}[section]
\newtheorem{lemma}[theorem]{Lemma}
\newtheorem{corollary}[theorem]{Corollary}
\newtheorem{proposition}[theorem]{Proposition}
\newtheorem{conjecture}[theorem]{Conjecture}
\theoremstyle{definition}
\newtheorem{definition}[theorem]{Definition}
\newtheorem{remark}[theorem]{Remark}

\newtheorem{example}[theorem]{Example}

\newtheorem{assumption}[theorem]{Assumption}

\newcommand{\R}{\mathbb{R}}
\newcommand{\bfx}{\mathbf{x}}
\newcommand{\bfp}{\mathbf{p}}
\newcommand{\bfq}{\mathbf{q}}
\newcommand{\C}{\mathbb{C}}
\newcommand{\PP}{\mathbb{P}}

\newcommand{\A}{\mathcal{A}}

\DeclareMathOperator{\Sym}{Sym} 

\title{Numerical Algebraic Geometry \\
for Energy Computations on Tensor Train Varieties}
\author{Viktoriia Borovik, Hannah Friedman, Serkan Ho\c{s}ten, Max Pfeffer}
\date{}

\begin{document}
\maketitle

\begin{abstract}
We study energy minimization problems in quantum chemistry  through the lens of computational algebraic geometry. We focus on minimizing the Rayleigh quotient of a Hamiltonian over a tensor train variety. 
The complex critical points of this problem 
approximate eigenstates of the quantum system, with the global minimum approximating the ground state.
We call the number of critical points the Rayleigh-Ritz degree.

We first study the Rayleigh-Ritz degree and introduce the Rayleigh-Ritz discriminant, which describes Hamiltonians that lead to a deficient number of critical points. We then specialize this framework to tensor train varieties: we identify instances when they are Segre products of projective spaces, report what we know about their defining ideals, and present a birational parametrization from products of Grassmannians. We use homotopy continuation to compute all critical points of this optimization problem over various tensor train and determinantal varieties. Finally, we use these results to benchmark state-of-the-art methods, the Alternating Linear Scheme and Density Matrix Renormalization Group.
\end{abstract}

\section{Introduction}

Since its derivation in 1926 by Erwin Schrödinger, the electronic Schrödinger equation has been studied extensively, mainly in order to predict molecular bindings between atoms or the structure of larger molecules. In numerical mathematics and theoretical physics/chemistry, the main task is to compute the ground state energy of the equation, that is, the smallest eigenvalue of the Hamiltonian operator. While this might seem straightforward, the problem size scales exponentially with the number of electrons involved, and therefore, standard eigensolvers fail even for small systems of $10$-$20$ electrons. Before the advent of crypto-currencies and AI, computations aiming at solving the Schrödinger problem commonly used about a third of the computation power of some large supercomputers \cite{Heinen2020}.

Many approximation methods for energy minimization of the electronic Schrödinger equation have been proposed by physicists and mathematicians, among them the {\em Hartree-Fock} method, which treats the involved electrons as uncorrelated, and {\em density functional theory},
where the two-particle term of the Hamiltonian is replaced by a one-particle potential. See \cite{Lin2019} for a mathematical introduction into both. For correlated many-body systems, widely used methods are the {\em coupled cluster} method \cite{Faulstich2024a,Faulstich2024b, Helgaker2014}, where a surrogate problem is solved on a lower-dimensional algebraic variety, and the {\em density matrix renormalization group (DMRG)} algorithm, which approximates the eigenstates by a low-rank tensor \cite{Szalay2015}. 
The latter was proposed by theoretical physicists along with {\em matrix product states}
to approximate solutions of the electronic Schrödinger equation in second quantization \cite{White1992}.
The realization that this approach is actually a low-rank tensor method emerged slowly and separately for physicists \cite{Schollwock2011} and mathematicians \cite{Oseledets2011}. In mathematics, the matrix product states have been dubbed {\em tensor trains (TT)} corresponding to the tensor train rank (TT rank) of a tensor, and the DMRG algorithm has been recognized as a two-site version of the {\em alternating linear scheme (ALS)} \cite{Holtz2012a}.
However, while this method often yields accurate results in practice (empirically validated using chemical measurements), theoretical guarantees for convergence and accuracy are lacking. 
In particular, the DMRG algorithm converges at best to local minima of the energy function.
The number of such local minima on the set of low-rank tensors as well as their approximation quality is unknown.

In this paper, our main focus is the minimization of the Rayleigh quotient 
\[
R_H(\psi) \;=\; \frac{\psi^T H \psi}{\psi^T \psi}
\]
for a Hamiltonian $H$, where $\psi$ is constrained to lie in a given tensor train variety. We first study the constrained Rayleigh quotient optimization problem from the perspective of algebraic geometry over arbitrary projective varieties. 
We define the {\em Rayleigh–Ritz degree} (RR degree) of a projective variety in Section~\ref{sec:RRdeg} as the number of complex critical points of this problem for a generic symmetric matrix $H$. 
Here, we present various ways of computing these critical points (Propositions~\ref{prop:Icrit-RR} and \ref{prop:RR degree-rational}). 
Proposition~\ref{prop:degree-of-L(V,H)} (\cite[Proposition 2.6]{SSW25}) gives an upper bound on the RR degree. 
In Section~\ref{sec:cor-dis}, we define the RR correspondence of a projective variety (Theorem~\ref{thm:RR-cor}) along with its parametric analog (Theorem~\ref{thm:parametric-RR-corr}).
We then define the RR discriminant, which is the set of symmetric matrices $H$ for which the critical point set does not have RR degree many isolated points.
In Section~\ref{sec:BW},
we follow \cite{SSW25} and interpret our optimization problem as a Euclidean distance optimization problem with respect to the Bombieri-Weyl inner product.  
This leads us to define the BW correspondence. We show that it is rational whenever the original variety is rational. This is our Corollary~\ref{cor:rational-BW}.

Next, we turn to tensor train varieties
in Section~\ref{sec: TTvarieties} and identify instances when they are Segre products of projective spaces (Theorem~\ref{thm:TT-as-Segre}). We also report what we know about the defining ideals of tensor trains. A birational parametrization of TT varieties from a product of Grassmannians will be presented in Section~\ref{sec:birational}.  
From this parametrization, we obtain the formula for the dimension of a TT variety (Corollary \ref{cor:TT-dim}).


We present our numerical experiments in Section~\ref{sec:numerics}. 
We compute all critical points of our optimization problem for small cases using the numerical algebraic geometry package \verb+HomotopyContinuation.jl+ \cite{hc}.
This enables us to test the correctness of the ALS and DMRG algorithms on these examples. 
In small cases, we observe that the ALS method can converge to any of the local minima and that the corresponding energy value can be far from optimal.

\section{Rayleigh-Ritz Optimization  and the RR degree}
\label{sec:RRdeg}
We start with the energy minimization problem arising in quantum chemistry. 
Given a 
matrix $H\in {\rm Sym}^2(\mathbb R^{n})$ and a complex projective algebraic
variety $V \subseteq \mathbb P^{n-1}$ not contained in the isotropic quadric $Q = \mathcal  V(\psi^T\psi)$,
we seek to solve the following optimization~problem:
\begin{align}
    \label{eqn:optproblem}
    &\textrm{minimize} \,\, \frac{\psi^T H \psi}{\psi^T\psi} \quad \, 
    \textrm{subject to} \,\, \psi \in V_\mathbb{R}.
\end{align}
Here $V_\mathbb R$ denotes the real locus of the complex projective variety $V$.
This is a relaxation of the problem ``Find eigenvectors of $H$ in a variety $V$''.
This problem is over-constrained, as the set of matrices that have eigenvectors in a fixed variety $V$ is a closed subvariety of ${\rm Sym}^2(\mathbb C^n)$. 
Such a set
is an example of a {\em nonlinear Kalman variety}; see \cite{Kalman}. 
\begin{remark}
    In 
    applications, the matrix $H$ is complex Hermitian rather than real symmetric.
    In this case the objective function is $\frac{\psi^\dagger H \psi}{\psi^\dagger\psi}$ where $^\dagger$ denotes the conjugate transpose. 
\end{remark}

When $V = \mathbb{P}^{n-1}$, the optimal value for \eqref{eqn:optproblem} is the smallest eigenvalue of $H$, and a corresponding eigenvector is an optimal solution. The rational function $\frac{\psi^T H \psi}{\psi^T\psi}$ is known as the {\it Rayleigh quotient}. 
When $V$ is a linear space, then \eqref{eqn:optproblem} 
is a {\it Rayleigh-Ritz optimization problem} over $V$ \cite{trefethen97}.  
We maintain this terminology in the nonlinear case.

A straightforward way of solving \eqref{eqn:optproblem} is by computing all of its {\it complex} critical points and identifying among them real critical points which minimize the value of the Rayleigh quotient. For us, a critical point is a nonsingular $\hat{\psi} \in V$ with $\hat{\psi}^T\hat{\psi} \neq 0$ such that the gradient $\nabla_\psi\left( \frac{\psi^T H \psi}{\psi^T\psi}\right)$ evaluated at $\hat{\psi}$ is in $(T_{\hat{\psi}}V)^\perp$. 
\begin{proposition} \label{prop:Icrit-RR} Let $V \subset \mathbb{P}^{n-1}$ be an irreducible variety of codimension $c$ with defining ideal $I_V  \subset \mathbb{R}[\psi_1, \ldots, \psi_n]$. The set of critical points of \eqref{eqn:optproblem} is the zero set of the critical ideal
 \begin{equation} 
\label{eqn:RR-critical-ideal}
I_{\mathrm{crit}}(V,H)  \!:=\!  \left(\! I_V \!+\! \bigg \langle (c+1)\mbox{-minors of} \left( \begin{array}{c}
(\psi^T\psi) \psi^T H - (\psi^T H \psi) \psi^T \\ \mathrm{Jac}(I_V) \end{array} \right)\!\bigg \rangle\!\right)\! :\!  ( I_{V_{\rm sing}} \cdot \psi^T \psi)^\infty     
\end{equation}
with $I_{V_{\mathrm{\rm sing}}}$  the ideal of the singular locus of $V$ and $\mathrm{Jac}(I_V)$  the Jacobian of the generators~of~$I_V$.
\end{proposition}
\begin{proof} The numerator of $\nabla_\psi(\frac{\psi^T H \psi}{\psi^T\psi})$ is $(\psi^T \psi)H \psi - (\psi^T H \psi) \psi$. Since we want $\psi^T \psi \neq 0$, the gradient is in $(T_\psi V)^\perp$ if and only if (the transpose of) this numerator is in the row span of the Jacobian $\mathrm{Jac}(I_V)$. This is equivalent to the vanishing of the $(c+1)$-minors of the above matrix in \eqref{eqn:RR-critical-ideal}. The saturation step guarantees that the solutions of $I_{\mathrm{crit}}(V,H)$ are indeed nonsingular points with $\psi^T\psi \neq 0$. 
\end{proof}

In Theorem~\ref{thm:RR-cor}, we will prove that the critical ideal $I_{\mathrm{crit}}(V,H)$ is a zero-dimensional ideal for generic $H$ and the number of critical points is a constant for all such $H$. Therefore, we define the \emph{Rayleigh-Ritz degree} (RR degree) of the variety $V$ to be the number of complex critical points of the optimization problem \eqref{eqn:optproblem} for general symmetric matrices $H$.
\begin{example} \label{ex:P1xP1}
The Segre surface $ V = \mathbb{P}^1 \times \mathbb{P}^1$ with defining ideal $I_V = \langle \psi_{00}\psi_{11} - \psi_{01}\psi_{10} \rangle \subset \mathbb{R}[\psi_{00}, \psi_{01}, \psi_{10}, \psi_{11}]$ will later appear as a tensor train variety, corresponding to ${\bf k}=(2,2)$ and ${\bf r}=(1)$. 
We use the symmetric matrix 
$$H = 
\begin{footnotesize}
    \begin{pmatrix} 100 & 78  & 76  & 42 \\
       78 & 170 & 111 &  67 \\
       76 & 111 & 85 &  54  \\
       42 & 67  & 54 &  41 \end{pmatrix}
\end{footnotesize}
 $$
and Proposition \ref{prop:Icrit-RR} to obtain $ I_{\mathrm{crit}}(V,H) $:
\begin{align*}
I_V + & \langle 
3307\psi_{00}\psi_{01}+5577\psi_{01}^2-6399\psi_{00}\psi_{10}-5016\psi_{10}^2\\
&\quad\quad \quad \quad\quad\quad \quad \quad\quad\quad \quad \quad\quad \quad -1188\psi_{00} \psi_{11}+783\psi_{01}\psi_{11}-4295\psi_{10}\psi_{11}-561\psi_{11}^2, \\
 & 6614\psi_{00}^2-16624\psi_{01}^2+24459\psi_{00}\psi_{10}+
      13582\psi_{10}^2\\
      &\quad\quad \quad \quad\quad\quad \quad \quad\quad\quad \quad \quad\quad\quad +606\psi_{00}\psi_{11}-14379\psi_{01}\psi_{11}+3978\psi_{10}\psi_{11}-3572\psi_{11}^2 \rangle.
\end{align*}
This ideal is a zero-dimensional  complete intersection and hence it has degree $8$. This implies that the RR degree of $\mathbb{P}^1 \times \mathbb{P}^1$ is $8$. 
Of the eight critical points, six are real, listed below as unit vectors, with the underlined one attaining the global minimum value 4.66885:
$$ \begin{footnotesize}
\begin{array}{ll}
(-0.05582,\, -0.43762,\, 0.11355 ,\, 0.89020 ) & 
(-0.49237 ,\, -0.35933 ,\, 0.64035,\, 0.46732 ) \\ 
 (-0.82393 ,\, 0.56305 ,\, -0.05284 ,\, 0.03611 ) &
 (0.55151 ,\, 0.59762 ,\, 0.39467 ,\, 0.42767 ) \\
 (-0.44158 ,\, 0.21247 ,\, 0.78549 ,\, -0.37795 ) &
 \underline{(-0.01170 ,\, 0.01709 ,\, -0.56495 ,\, 0.82486).}
 \end{array}
 \end{footnotesize}
$$
\end{example}

The RR degrees of Segre-Veronese varieties were recently derived in \cite{SSW25}, where, 
in particular, explicit formulae for the RR degrees of rank-one matrices of any size and of rank-one binary tensors were obtained.
\begin{proposition}[{\cite[Corollaries 5.2 and 5.3]{SSW25}}] 
    The RR degree of the variety of rank-one binary tensors of order $n$ $(V = (\mathbb P^1)^n)$ is $2^nn!$. The RR degree of the variety of rank-one $n \times m$ matrices $(V = \mathbb P^{n-1} \times\mathbb P^{m-1})$ with $n \leq m$ is 
$\sum_{i = 1}^n 4^{i-1}\binom{n}{i}\binom{m}{i}$.
\end{proposition}

\subsection{The Lagrangian Locus}
We introduce the Lagrangian locus of a projective variety and a symmetric matrix $H$. 
This yields another method for computing the critical points of \eqref{eqn:optproblem}, and gives an upper bound on the RR degree of a variety. 

In the above, the \emph{critical locus} consists of points where the gradient of the level set ${\frac{\psi^T H \psi}{\psi^T \psi} = \lambda}$ is in the span 
of the gradients of the forms defining $V$. These level sets correspond to a pencil of quadrics $\{\mu_1 \psi^T H \psi - \mu_2 \psi^T\psi =0\}$ 
with $[\mu_1 \, : \, \mu_2] \in \mathbb{P}^1$; see \cite[Section~2.2]{Ran}. The gradient vector of this pencil 
at the point $\hat{\psi} \in \mathbb{P}^{n-1}$ is the normal vector to a hyperplane~$\Gamma_{\hat{\psi}}$. This way we get a 
pencil of hyperplanes 
$\{\mu_1 (H\hat{\psi})^T\psi - \mu_2 (\hat{\psi})^T \psi  = 0\}$. Equivalently, this pencil can be characterized as a family of hyperplanes containing the codimension two linear subspace $\Lambda_{\hat{\psi}}$ given by $(H\hat{\psi})^T \psi = (\hat{\psi})^T \psi = 0$. 
Let $V_{\rm reg} = V\backslash V_{\rm sing}$ be the regular locus of~$V$.

\begin{definition} [{\cite[Definition 2.5]{Ran}}] Consider the incidence 
$$\mathcal{I}(V,H) \, := \, \overline{\{(\psi, \Gamma) \in V_{\mathrm{reg}} \times (\mathbb{P}^{n-1})^* \, : \, T_\psi V \subset \Gamma, \, \, \Lambda_\psi \subset \Gamma\}}$$ 
and let $L(V,H) = \pi_1(\mathcal{I}(V,H))$ where $\pi_1 \, : \,\mathcal I(V,H) \to V$ is the projection onto the first factor. We will call $L(V,H)$ the \emph{Lagrangian locus} of the Rayleigh-Ritz optimization problem \eqref{eqn:optproblem}. 
\end{definition}

\begin{proposition} \label{prop:Lagrangian-locus}
Let $V \subset \mathbb{P}^{n-1}$ be an irreducible algebraic  variety of codimension $c$ and let $I_V \subset \mathbb{R}[\psi_1, \ldots, \psi_n]$ be its defining ideal. Then the Lagragian locus is equal to $$L(V,H) = \, \overline{\{ \psi \in V_{\mathrm{reg}} \, : \, 
\mathrm{rank}\left( \overline{\mathrm{Jac}}(\psi) \right) \leq c+1 \}},$$
where $\overline{\mathrm{Jac}}(\psi)$ is the augmented Jacobian
\begin{equation}\label{eqn:lagrange-mat}
\begin{pmatrix} \psi^T \\ \psi^TH \\ \mathrm{Jac}(I_V)
\end{pmatrix}.
\end{equation}
Therefore $L(V,H)$ is the zero set of the ideal
 $\left( I_V + \langle (c+2) \mbox{-minors of } \overline{\mathrm{Jac}}(\psi) \rangle \right) \, : \,  I_{V_{\rm sing}}^\infty$.
\end{proposition}
\begin{proof}
A $\hat{\psi} \in V_{\mathrm{reg}}$ is in  
$\pi_1(\mathcal{I}(V,H))$ 
if and only if there exists a normal vector that is both in the row span of $\mathrm{Jac}(I_V)$ evaluated at $\hat{\psi}$ and in the span of $\hat{\psi}^T$ and $\hat{\psi}^T H$. This happens if and only if $\overline{\mathrm{Jac}}(\hat{\psi})$ has rank at most $c+1$ which is equivalent to $\hat{\psi} \in L(V,H)$. 
\end{proof}

\begin{lemma} \label{lem:lagrange-crit}
For a fixed $H$, the Lagrangian locus contains the critical locus.   
\end{lemma}
\begin{proof}
Let $\hat{\psi}$ be in the critical locus given by $I_{\mathrm{crit}}(V,H)$. Since $\hat{\psi}^T \hat{\psi} \neq 0$, the rank condition in \eqref{eqn:RR-critical-ideal} implies that $\overline{\mathrm{Jac}}(\hat{\psi})$ has  rank at most $c+1$. Therefore $\hat{\psi}$ is contained in $L(V,H)$. 
\end{proof}

\begin{example} \label{ex:lagrange-p1p1}
We compute the Lagrangian locus of the Rayleigh-Ritz optimization problem~\eqref{eqn:optproblem} for $V=\mathbb{P}^1 \times \mathbb{P}^1$ using the matrix $H$
from Example \ref{ex:P1xP1}. This is given by $I_V$ together with the $3$-minors of the 
augmented Jacobian 
$$
\begin{footnotesize}
\begin{pmatrix}
 \psi_{00} & \psi_{01} & \psi_{10} & \psi_{11} \\
 (H\psi)_{00} & (H\psi)_{01} & (H\psi)_{10} & (H\psi)_{11} \\
 \psi_{11} & -\psi_{10} & -\psi_{01}  & \psi_{00}
\end{pmatrix}
\end{footnotesize}.
$$
The Lagrangian locus $L(V, H)$ is zero-dimensional and has degree $12$. The primary decomposition of the defining ideal has three components over $\mathbb Q$:
$$ I_{\mathrm{crit}}(V,H)  \cap \langle \psi_{01} - \psi_{10}, \psi_{00} + \psi_{11}, \psi_{10}^2 + \psi_{11}^2 \rangle \cap \langle 
\psi_{01} + \psi_{10}, \psi_{00} -\psi_{11}, \psi_{10}^2 + \psi_{11}^2 \rangle. $$
Note that both of the last two components contain the isotropic quadric $\psi_{00}^2 + \psi_{01}^2 + \psi_{10}^2 + \psi_{11}^2$. Therefore the four isotropic points they contribute are not critical points. 
\end{example}
\begin{example}\label{ex:lagrange-twisted-cubic}
The Lagrangian locus of the Rayleigh-Ritz problem \eqref{eqn:optproblem} for the twisted cubic $C$ with the same $H$ as in Example \ref{ex:P1xP1} is obtained from the $4$-minors of the corresponding  augmented Jacobian. The resulting ideal is exactly the critical ideal $I_{\mathrm{crit}}(C,H)$ of degree $10$. 
\end{example}

We now specify the genericity condition that ensures that the critical and Lagrangian loci are equal.
\begin{theorem}\label{thm:lagrange=crit}
    The critical locus of \eqref{eqn:optproblem} equals the closure $\overline{L(V,H)\backslash Q}$ of the Lagrangian locus after removing the isotropic quadric $Q  := \mathcal{V}(\psi^T\psi)$.
    The Lagrangian locus contains all points at which $V$ intersects $Q$ nontransversely for all $H$. 
    For generic $H$, every isotropic point in $L(V,H)$ is a singular point of $V \cap Q$.
\end{theorem}
\begin{proof}
By Lemma~\ref{lem:lagrange-crit}, the critical locus is contained in the Lagrangian locus.
Conversely, all points $\hat{\psi} \in L(V,H)$ with  $\hat{\psi}^T\hat{\psi} \neq 0$ are in the corresponding critical locus. 
Indeed, a linear combination $\hat{\mu}_1 \hat{\psi}^T H - \hat{\mu}_2 \hat{\psi}^T$ is in the row span of $\mathrm{Jac}(I_V)$ evaluated 
at $\hat{\psi}$. If $I_V = \langle f_1, \ldots f_t\rangle$ with $\deg(f_i) = d_i$, this means that
$$\hat{\mu}_1 \hat{\psi}^T H - \hat{\mu}_2 \hat{\psi}^T = 
\sum_{i=1}^t \lambda_i (\nabla f_i)(\hat{\psi})$$
for some $\lambda_1,\dots,\lambda_t \in \mathbb{C}$. Here, $(\nabla f_i)(\hat{\psi})$ is the $i$-th row of $\mathrm{Jac}(I_V)$ evaluated at the point $\hat{\psi}$. Multiplying this identity on the right by $\hat{\psi}$, we obtain
$$\hat{\mu}_1 \hat{\psi}^T H \hat{\psi}- \hat{\mu}_2 \hat{\psi}^T \hat{\psi} = 
\sum_{i=1}^t \lambda_i \left(\sum_{j=1}^n \frac{\partial f_i}{\partial{\psi_j}}(\hat{\psi}) \hat{\psi}_j \right) = 0,$$
where the last equality holds because $\sum_{j=1}^n \frac{\partial f_i}{\partial{\psi_j}}(\hat{\psi}) \hat{\psi}_j$
is equal to $\sum_{j=1}^n \frac{\partial f_i}{\partial{\psi_j}} \psi_j = d_i f_i$ evaluated at $\hat{\psi}$. 
We conclude that $\frac{\hat{\mu}_2}{\hat{\mu}_1} = \frac{(\hat{\psi})^T H \hat{\psi}}{(\hat{\psi})^T \hat{\psi}}$ and 
therefore $(\hat{\psi}^T\hat{\psi}) \hat{\psi}^T H - (\hat{\psi}^T H \hat{\psi}) \hat{\psi}^T $
is in the row span of $\mathrm{Jac}(I_V)$ evaluated at $\hat{\psi}$. 
This proves the first statement of the theorem.

For the second statement,  it is enough to show that if the isotropic quadric $Q$ and $V$
 intersect nontransversally at $\hat \psi$, then $\hat \psi \in L(V, H)$.
 Since $V$ intersects $Q$ nontransversally at $\hat \psi$,  the row vector $\hat{\psi}^T$ is in the row span of $\mathrm{Jac}(I_V)$ evaluated at $\hat{\psi}$. This implies that $\overline{\mathrm{Jac}}(\hat{\psi})$ has rank at most $c+1$, and hence $\hat{\psi} \in L(V,H)$. 
 
 For the third claim, suppose that $\hat{\psi}$
 with $\hat{\psi}^T\hat{\psi} = 0$ be in $L(V,H)$. 
Similar to the discussion above, one can show that  $\hat{\mu}_1 \hat{\psi}^T H \hat{\psi} - \hat{\mu}_2 \hat{\psi}^T \hat{\psi} = \hat{\mu}_1 \hat{\psi}^T H \hat{\psi} = 0$. 
For generic $H$, this implies that $\hat{\mu}_1= 0$
and we conclude that $\hat{\psi}^T$ is in the row span of $\mathrm{Jac}(I_V)$ evaluated at $\hat{\psi}$. In other words, $Q$ and $V$ intersect at $\hat{\psi}$ nontransversally.  
 \end{proof}
\begin{example}\label{ex:P1P1excess}
The twisted cubic intersects the quadric $Q$ transversally, and therefore the critical locus is equal to the Lagrangian locus, as we saw in Example~\ref{ex:lagrange-twisted-cubic}. 
In contrast, the intersection of $\mathbb{P}^1 \times \mathbb{P}^1$ with $Q$ contains four
nontransversal intersection points:
$$ [-1:i:i:1], \; [-1:-i:-i:1], \; [1:i:-i:1], \; [1:-i:i:1].$$
These four extra points are present in the Lagrangian locus $L(\mathbb{P}^1 \times \mathbb{P}^1,H)$ for any $H$.
They come from the extraneous components in the primary decomposition in Example~\ref{ex:lagrange-p1p1}.
\end{example}

\begin{proposition}[{\cite[Theorem 3.1]{SSW25}}] \label{prop:degree-of-L(V,H)} Let $V \subset \mathbb{P}^{n-1}$ be an irreducible projective variety of codimension $c$ and let $I_V = \langle f_1, \ldots, f_t\rangle \subset \mathbb{R}[\psi_1, \ldots, \psi_n]$ be its defining ideal. If $f_1, \ldots, f_c$ form a regular sequence with respective degrees $d_1, \ldots, d_c$, then the RR degree of $V$ is at most  
\begin{equation}\label{eqn:GPT}
 (d_1 \cdots d_c) \sum_{i_0+i_1+\cdots+i_c = n-c-1} (i_0+1)(d_1-1)^{i_1}(d_2-1)^{i_2} \cdots (d_c-1)^{i_c}. 
 \end{equation}
The equality holds if $V$ is a generic complete intersection.
\end{proposition}
For completeness, we include a sketch of the proof leaving out the details about intersection theory; for these, see \cite[Theorem 14.4]{Ful} and \cite[Chapter 12]{EH}. 
\begin{proof}
The variety $Y$ defined by $f_1, \ldots, f_c$ is a complete intersection, and $V$ has to be an irreducible component of $Y$. 
Thus RR degree of $V$ is at most the RR degree of $Y$.
Note that these RR degrees are finite by Theorem~\ref{thm:RR-cor}.
The RR degree of $Y$ is, in turn, bounded above by the RR degree of a generic complete intersection of codimension $c$. 
It therefore suffices to consider the case of a generic complete intersection.

For a generic complete intersection $V$, the Lagrangian locus is equal to the critical locus by Theorem~\ref{thm:lagrange=crit}. 
Furthermore, the ideal of the Lagrangian locus can be computed without saturation: it is generated by $f_1, \ldots, f_c$ together with the $(c + 2)$-minors of 
$$\overline{\mathrm{Jac}}(\psi) = 
\begin{footnotesize}
\begin{pmatrix} \psi_1 & \psi_2 & \cdots & \psi_n \\
(H\psi)_1 & (H\psi)_2 & \cdots & (H\psi)_n \\
\frac{\partial f_1}{\partial \psi_1} & \frac{\partial f_1}{\partial \psi_2} & \cdots &\frac{\partial f_1}{\partial \psi_n} \\
\frac{\partial f_2}{\partial \psi_1} & \frac{\partial f_2}{\partial \psi_2} & \cdots &\frac{\partial f_2}{\partial \psi_n} \\
\vdots & \vdots & \cdots & \vdots \\
\frac{\partial f_c}{\partial \psi_1} & \frac{\partial f_c}{\partial \psi_2} & \cdots &\frac{\partial f_c}{\partial \psi_n} 
\end{pmatrix}\end{footnotesize}.
$$
 By the Giambelli-Thom-Porteous formula \cite[Theorem 14.4]{Ful} and Bezout's Theorem the RR degree of $V$ is \eqref{eqn:GPT}.
\end{proof}

\begin{example}
  By Example \ref{ex:lagrange-twisted-cubic}, the Lagrangian and critical loci agree for the twisted cubic, and the RR degree is $10$. 
  If we use the formula in Proposition \ref{prop:degree-of-L(V,H)} as an upper bound (with $n=4$, $c=2$, and $d_1=d_2=2$), we obtain a bound of $16$. 
  The discrepancy comes from the fact that the twisted cubic is not a complete intersection.
\end{example}
\begin{corollary}[{\cite[Corollary 3.3]{SSW25}}]\label{cor:hypersurface}
Let $V \subset \mathbb{P}^{n-1}$ be a hypersurface of degree $d$ and codimension $c$. Then the RR degree of $V$ is at most
$$d \sum_{j=1}^{n-c} j (d-1)^{n-c-j}.$$
Moreover, equality holds for a general hypersurface $V$.
\end{corollary}
\begin{example}
$\mathbb{P}^1 \times \mathbb{P}^1$ is a hypersurface of degree two in its Segre embedding in $\mathbb{P}^3$. Therefore the Corollary \ref{cor:hypersurface} gives the upper bound of $12$ for the RR degree. However, this determinantal surface with RR degree $8$ is not general enough. 
The discrepancy comes from the four nontransverse intersection points in Example~\ref{ex:P1P1excess}.
\end{example}

\subsection{Parametric Critical Locus}\label{sec:param-crit-loc}
We now consider the case when $V \subseteq \mathbb{P}^{n-1}$ admits a birational parametrization $\psi \, : U \dashrightarrow V$ where $U$ is either an affine space or a product of projective spaces.  
Note that this includes tensor train varieties with the parametrization described in Section~\ref{sec:birational}.
We consider a parametric version of the Rayleigh-Ritz optimization problem \eqref{eqn:optproblem}, namely,
\begin{equation}\label{eqn:opt-param}
    \textrm{minimize  } R_H(x) = \frac{\psi(x)^TH\psi(x)}{\psi(x)^T\psi(x)}.
\end{equation}
This is an unconstrained optimization problem, and its critical points are precisely  the values of $x \in \psi^{-1}(V_{\rm reg} \backslash Q)$ for which $\nabla_x R_H(x) = 0$. 
If $x$ is critical for \eqref{eqn:opt-param}, then its image $\psi(x)$ is critical for \eqref{eqn:optproblem} if the Jacobian of the parametrization $\mathrm{Jac}(\psi)$ defines a surjective map onto the tangent space $T_{\psi(x)} V$ of $V$ at $\psi(x)$; see \cite[Lemma 3.1]{FH}.
We therefore fix the notation $I_{\psi, s}$ for the ideal generated by $s$-minors of the Jacobian $\mathrm{Jac}(\psi)$. 

\begin{proposition}
\label{prop:RR degree-rational}
Let $V\subseteq \mathbb{P}^{n-1}$ be a variety with a  birational parametrization $\psi \, : U \dashrightarrow V$ and base locus $B$.  
Then the set of critical points of \eqref{eqn:opt-param} is the zero set of the critical ideal
$$ J_{\mathrm{crit}}(V,H) = \bigg \langle (\psi(x)^T\psi(x))^2 \cdot \nabla_x R_H(x) \bigg \rangle : (\langle \psi(x)^T \psi(x) \rangle \cdot I_B \cdot I_{\psi, n - c})^\infty.$$
For a general $H \in \mathrm{Sym}^2(\mathbb{C}^n)$, the number of critical points is equal to the RR degree of $V$.
\end{proposition}
\begin{proof}
Critical points are exactly the points at which the gradient 
$\nabla_x R_H(x)$ vanishes. 
By the quotient rule, the partial derivatives of $R_H(x)$ have a common denominator of $(\psi(x)^T\psi(x))^2$.
We multiply by this factor to clear  denominators, so the ideal is generated by polynomials. 
We then remove points where $\psi(x)^T\psi(x) = 0$ by saturating. The last statement follows from Theorem \ref{thm:parametric-RR-corr} below. 
\end{proof}

\begin{example}
 The twisted cubic $C \subset \mathbb{P}^3$ is parametrized by $[a \, : \, b] \mapsto [a^3\, : \, a^2b \, : \, ab^2 \, : \, b^3]$.
 Using Proposition \ref{prop:RR degree-rational} and the matrix $H$ in Example \ref{ex:P1xP1},  we compute $J_{\mathrm{crit}}(C,H)$ to be 
 {\small 
 $$ \langle 78a^{10}+222a^9b+381a^8b^2+238a^7b^3+189a^6b^4-280a^5b^5-381a^4b^6-562a^3b^7-405a^2b^8-178ab^9-54b^{10}\rangle. $$} We conclude that the RR degree of the twisted cubic is $10$. In this example, there are two real critical points obtained from the parameter values 
 \[[0.52795 \, : \, -0.94040] \quad \text{and} \quad [0.83400 \, : \,  0.74797].\]
 The first one gives the global minimum with the corresponding value $14.9845$.
\end{example}

Similarly, the Lagrangian locus  has a parametric analog. 
As above, let $\psi: U   \dashrightarrow V$ be a birational parametrization of a variety $V$.  Suppose that the parametrization space $U$ has (projective or affine) coordinates $x_1, \ldots, x_m$. 
Write $f(x)$ for the numerator $\psi(x)^TH\psi(x)$ of $R_H(x)$ and $g(x) = \psi(x)^T\psi(x)$ for the denominator. 
By the quotient rule, the derivative with respect to $x_i$ vanishes precisely when $g(x)\frac{\partial f}{\partial x_i}(x) - f(x)\frac{\partial g}{\partial x_i}(x)$ equals zero. 
Therefore the critical points of \eqref{eqn:opt-param} are cut out by the $2$-minors of 
\begin{equation}\label{eqn:2-row-mat}
    \begin{pmatrix}
        f & \frac{\partial f}{\partial x_1} & \cdots & \frac{\partial f}{\partial x_m}\\
        g & \frac{\partial g}{\partial x_1} & \cdots & \frac{\partial g}{\partial x_m}
    \end{pmatrix} = 
    \begin{pmatrix}
        \psi(x)^T\\
      \psi(x)^TH
    \end{pmatrix}
    \cdot 
    \begin{pmatrix}
        \psi(x) & \mathrm{Jac}(\psi)(x)
    \end{pmatrix}.
\end{equation}
We remark that if $U$ is a projective space, then the first column of \eqref{eqn:2-row-mat} can be omitted using Euler's relation.

\begin{proposition}\label{prop:param-lagrange}
    Let $\psi: U \dashrightarrow V$ be a birational parametrization of a variety $V \subseteq \mathbb{P}^{n-1}$ of codimension $c$.
    Then the parametric Lagrangian locus is equal to 
    \begin{equation}\label{eqn:lagrange-param}
    \psi^{-1}(L(V, H)) = \overline{
        \{x \in \psi^{-1}(V_{\rm reg}) \backslash V(I_{\psi, n-c}):  
        {\rm rank} \,\,
        \textrm{\eqref{eqn:2-row-mat}} \leq 1.\}}
    \end{equation}
\end{proposition}
As in the implicit case, we obtain the critical locus after saturating by isotropic quadric.
\begin{corollary}\label{cor:param-crit=lagrange}
    The parametric critical locus equals the closure of the parametric Lagrangian locus \eqref{eqn:lagrange-param}
    after removing the isotropic quadric $Q$.
\end{corollary}
\begin{proof}
    By the quotient rule, a point $\hat{x} \in U$ is a critical point of \eqref{eqn:opt-param} if and only if $g(x) =\psi(\hat x)^T\psi(\hat x)$ is nonzero and \eqref{eqn:2-row-mat} has rank $1$.
\end{proof}

\begin{example}
    We parametrize $\mathbb P^1 \times \mathbb P^1 \subseteq \mathbb P^3$ by $([a:b], [c:d]) \mapsto [ac: ad: bc: bd]$.
    For a given $H$, the ideal of the parametric Lagrangian locus is generated by $2$-minors of the matrix
    \begin{align*}
    \begin{footnotesize}
       \begin{pmatrix}
           c & d & &\\
            & & c & d \\
           a&  & b & \\
           & a &  & b
       \end{pmatrix} \cdot 
       \begin{pmatrix}
           ac & h_{11}ac + h_{12}ad + h_{13}bc + h_{14}bd\\
           ad & h_{12}ac + h_{22}ad + h_{23}bc + h_{24}bd\\
           bc & h_{13}ac + h_{23}ad + h_{33}bc + h_{34}bd\\
           bd & h_{14}ac + h_{24}ad + h_{34}bc + h_{44}bd
       \end{pmatrix} \end{footnotesize}.
    \end{align*}
    This ideal decomposes as $J_{\rm crit}(V, H)  \cap \langle a^2 + b^2, c^2 + d^2 \rangle$. 
    The excess component consists of the four points $([1:i], [1:i]),\ ([1:i], [1:-i]),\ ([1:-i], [1:i]), $ and $([1:-i], [1:-i])$. 
    These are precisely the four points in Example~\ref{ex:P1P1excess}.
    
For the twisted cubic, the parametric Lagrangian locus agrees with the parametric critical locus as in Example~\ref{ex:P1P1excess}.
\end{example}

\section{Rayleigh-Ritz Correspondence and Discriminant}\label{sec:cor-dis}
In this section, we allow the symmetric data matrix $H$ to vary.
Because the Rayleigh quotient $\frac{\psi^TH\psi}{\psi^T\psi}$ is linear in the entries of $H$, the problem \eqref{eqn:optproblem} has the same critical points if $H$ is replaced with any nonzero scalar multiple $\lambda H$.
We therefore only consider the data matrices up to scaling, and accordingly work in $\mathbb P({\rm Sym}^2(\mathbb C^n))$.
We first define the Rayleigh-Ritz correspondence as the incidence variety of symmetric matrices $H$ and the corresponding critical points $\psi$.
We define this correspondence for general projective varieties. 
We illustrate this construction with small examples, namely projective spaces, rational normal curves and the smallest tensor train variety, $\mathbb P^1 \times \mathbb P^1$. 
Similar correspondences have been introduced for maximum likelihood estimation \cite{HS14, KSSW} and for Euclidean distance optimization \cite{DHGSR}.

We then define the Rayleigh-Ritz discriminant, which describes the matrices for which the critical locus of \eqref{eqn:optproblem} exhibits nongeneral behavior. 
Analogous discriminants were defined in \cite{KKT25} for the maximum likelihood estimation problem and in \cite{DHGSR} for the Euclidean distance optimization problem.

\subsection{Rayleigh-Ritz Correspondence}
Throughout this section, $V \subseteq \mathbb P^{n-1}$ represents an irreducible projective variety not contained in the isotropic quadric $Q = V(\psi^T\psi)$.
We define the {\em Rayleigh-Ritz correspondence} (RR correspondence) as the pairs of smooth points $\psi \in V_{\rm reg}$ and symmetric matrices $H$ such that $\psi$ is a critical point of \eqref{eqn:optproblem} for that matrix $H$: 
\begin{equation*}\label{eqn:RR-cor}
    \mathcal R_V = \overline{\{(\psi, H) \in V_{\rm reg} \times \mathbb P({\rm Sym}^2(\mathbb C^n)): 
    \psi \textrm{ is a critical point of \eqref{eqn:optproblem}}
    \}}.
\end{equation*}
Here both the vector $\psi$ and the data matrix $H$ are indeterminates.
If $V$ has codimension $c$, then by Proposition~\ref{prop:Lagrangian-locus} and Theorem~\ref{thm:lagrange=crit} the correspondence is the zero locus of the ideal
\begin{equation}\label{eqn:IRV}
    \big \langle
    (c + 2)\textrm{-minors of \eqref{eqn:lagrange-mat}} 
    \big \rangle :
    (I_{V_{\rm sing}} \cdot \langle \psi^T\psi \rangle)^\infty.
\end{equation}
If $V$ is smooth and intersects the quadric $Q$ transversally, then this ideal can be computed without saturation by Theorem~\ref{thm:lagrange=crit}. 

\begin{theorem}\label{thm:RR-cor}
    The RR correspondence $\mathcal R_V$ of an irreducible variety $V \subseteq \mathbb P^{n-1}$ of codimension $c$ is an irreducible variety in $\mathbb P^{n-1} \times \mathbb P({\rm Sym}^2(\mathbb C^n))$ of dimension $\binom{n+1}{2}-1$. 
    The first projection $\pi_1: \mathcal R_V \to V \subseteq \mathbb P^{n-1}$ is a vector bundle of rank $\binom{n+1}{2} - n + c$ over $V_{\rm reg} \backslash Q$. 
    Over general data matrices $H \in \mathbb P({\rm Sym}^2(\mathbb C^n))$, the second projection $\pi_2: \mathcal R_V \to \mathbb P({\rm Sym}^2(\mathbb C^n))$ has  finite fibers $\pi_2^{-1}(H)$ of cardinality equal to the RR degree of $V$.
\end{theorem}
\begin{proof}
    To prove the vector bundle property, consider a point 
    $\hat \psi \in V_{\rm reg} \backslash Q$.  
    By \eqref{eqn:IRV}, $(\psi,H) \in \pi_1^{-1}(\psi)$ precisely when  $\hat \psi^T H$ is in the $c$-dimensional subspace of $\mathbb P^{n-1}$ spanned by $\hat \psi^T$ and the rows of ${\rm Jac}(I_V)$.
    The linear map $H \mapsto \psi^TH$ is surjective and hence its fibers have dimension $\binom{n + 1}{2} -n$. 
    Therefore the projective linear space $\pi_1^{-1}(\psi)$ has dimension $\binom{n+1}{2} - n + c$.
    Since $V$ is irreducible, the correspondence $\mathcal R_V$ is irreducible as well. 
    The projective dimension of $\mathcal R_V$ is 
    $\binom{n+1}{2}-1$.
    Since $\pi_2$ is finite, the general fibers $\pi_2^{-1}(H)$ are finite and have the same size. 
    This cardinality is the RR degree of $V$.
\end{proof}
\begin{example}
    The RR correspondence of the rational quadric $V(\psi_1\psi_3 - \psi_2^2) \subseteq \mathbb P^2$ is a 5-dimensional subvariety of $\mathbb P^2 \times \mathbb P^5$ and is minimally generated by $\psi_1\psi_3 - \psi_2^2$ and 
    \begin{align*}
        \psi_1^3h_{12} + \psi_1^2\psi_2(-h_{11} + 2h_{13} + h_{22}) + \psi_1^2\psi_3(-h_{12} + 3h_{23} - 2h_{11}) + 2\psi_1\psi_2\psi_3h_{33}\\
        + \psi_1\psi_3^2(-3h_{12} + h_{23} - 2h_{13})
        + \psi_2\psi_3^2(-2h_{13} - h_{22} + h_{33}) 
        - \psi_3^3h_{23}.
    \end{align*}
\end{example}
We turn now to the parametric situation introduced in Section~\ref{sec:RRdeg}. 
Suppose $V$ admits a birational parametrization $\psi: U \dashrightarrow V$ from an affine space or product of projective spaces $U$.
We define the parametric RR correspondence as 
\begin{equation*}
    \mathcal{P}_V = \overline{\{(x, H) \in \psi^{-1}(V_{\rm reg} \backslash Q)
\times \mathbb P({\rm Sym}^2(\mathbb C^n)) : 
     x \textrm{ is a critical point of \eqref{eqn:opt-param}}
    \}}.
    \end{equation*}
Since the parametric and implicit critical points agree, 
$$\psi \times {\rm Id}:  U \times \mathbb P( {\rm Sym}^2(\mathbb C^n) )\dashrightarrow \mathbb P^{n-1} \times \mathbb P( {\rm Sym}^2(\mathbb C^n) )$$
is a birational map from the parametric correspondence $\mathcal P_V$ to the usual correspondence $\mathcal R_V$.
By Corollary~\ref{cor:param-crit=lagrange}, the parametric correspondence is the zero set of the ideal
\begin{equation*}
   \big \langle 2\textrm{-minors of }
   {\rm Jac}(\psi)^T \cdot 
   \begin{pmatrix} 
   \psi(x) & H\cdot \psi(x)
   \end{pmatrix}
   \big \rangle : 
   (\langle (n-c)\textrm{-minors of } {\rm Jac}(\psi) \rangle \cdot 
   \langle \psi(x)^T\psi(x)\rangle \cdot I_B)^\infty
\end{equation*}
where $I_B$ is the ideal of the base locus $B$ of $\psi$.

\begin{theorem}\label{thm:parametric-RR-corr}
    The parametric RR correspondence $\mathcal P_V$  of $V = \overline{{\rm im}(\psi)} \subseteq \mathbb P^{n-1}$ is an irreducible variety of dimension $\binom{n+1}{2} - 1$. 
    The first projection $\gamma_1: \mathcal P_V \to U$ is a vector bundle of rank $\binom{n+1}{2} -n + c$ over $\psi^{-1}(V_{\rm reg}\backslash Q)$. 
    Over general data matrices $H \in {\rm Sym}^2(\mathbb C^n)$, the second projection $\gamma_2: \mathcal P_V \to \mathbb{P}({\rm Sym}^2(\mathbb C^n))$ has  finite fibers $\gamma_2^{-1}(H)$ of cardinality equal to the RR degree of $V$.
\end{theorem}
\begin{proof}
    Since the parametric and implicit critical points agree, by the definitions of $\mathcal P_V$ and $\mathcal R_V$, 
    we have the commutative diagram 
\[
        \begin{tikzcd}
            \mathcal P_V \arrow[r, dashed, "\psi \times {\rm Id}"] \arrow[d, "\gamma_1"] & \mathcal R_V \arrow[d, "\pi_1"]\\
            U \arrow[r, dashed, "\psi"] & \mathbb P^{n-1}
        \end{tikzcd}
\]
    Therefore $\gamma_1 = \psi^{-1} \circ \pi_1 \circ (\psi \times {\rm Id})$ on pairs $(\hat x, \hat H) \in \mathcal P_V$.
    Such an $\hat x$ has fiber $\gamma_1^{-1}(\hat x) = (\psi^{-1} \times {\rm Id})(\pi_1^{-1}(\psi(\hat x))) = \{\hat x\} \times \pi_2(\pi_1^{-1}(\psi(\hat x)))$.
    Hence the fibers of $\gamma_1$ are isomorphic to the fibers of $\pi_1$ and 
    the vector bundle property, irreducibility, and dimension all follow from Theorem~\ref{thm:RR-cor}.
    For the second projection, we have the commutative diagram 
    \[
        \begin{tikzcd}
            \mathcal P_V \arrow[r, dashed, "\psi \times {\rm Id}"] \arrow[dr, "\gamma_2"] & \mathcal R_V \arrow[d, "\pi_2"]\\
             & \mathbb P({\rm Sym}^2(\mathbb C^n))
        \end{tikzcd}
\]
    Thus $\gamma_2^{-1}(\hat H) = (\psi \times {\rm Id})^{-1}(\pi_2^{-1}(\hat H))$. 
    Since $\pi_2^{-1}(\hat H)$ is generically a finite set of cardinality ${\rm RR degree}(V)$ by Theorem~\ref{thm:RR-cor} and $\psi \times {\rm Id}$ is generically injective, the last statement holds. 
\end{proof}

\begin{example}
    The parametric RR correspondence of the rational quadric curve parametrized as $[a:b] \mapsto [a^2:ab:b^2]$
    is a hypersurface in $\mathbb P^1 \times \mathbb P^5$ cut out by 
    \begin{align*}
        a^6 h_{12} + 
        a^5b (h_{11} + 2h_{13} + h_{22}) +
        a^4b^2(-h_{12} + 3h_{23}) &+ 
        a^3b^3(2h_{11} + 2h_{33}) \\+ 
        a^2b^4(-3h_{12} + h_{23}) &+ 
        ab^5(-2h_{13} - h_{22} + h_{33}) - 
        b^6h_{23}.
    \end{align*}
\end{example}
\begin{example}
    The parametric RR correspondence of $\mathbb P^1 \times \mathbb P^1$ is a 9-dimensional subvariety of $(\mathbb P^1 \times \mathbb P^1) \times \mathbb P^9$ of degree 17. 
    Its prime ideal is minimally generated by eight polynomials: two of degree $(2,2,1)$, two of degree $(1,4,2)$, two of degree $(4,1,2)$, one of degree $(0,8,4)$ and one of degree $(8,0,4)$; see \cite{zenodo_our} for these polynomials. 
\end{example}

\subsection{Rayleigh-Ritz Discriminant}\label{sec:RR-dis}
In Theorem~\ref{thm:RR-cor}, we proved that for a general choice of data matrix  $H$, the fiber $\pi_2^{-1}(H)$ has RR degree many points. 
By ``general,'' we mean that this statement holds outside of a subvariety of the data space $\mathbb P({\rm Sym}^2(\mathbb C^n))$.
We call this subvariety, which is typically a hypersurface, the {\em Rayleigh-Ritz discriminant} (RR discriminant):
\begin{equation*}
    \Sigma_V = \overline{\{
    H \in \mathbb P({\rm Sym}^2(\mathbb C^n)) : 
    \textrm{ the critical locus of \eqref{eqn:optproblem} is infinite or nonreduced}
    \}}.
\end{equation*}

We use the Jacobian criterion to determine when the critical locus of a matrix $H$ is singular by treating the variables $h_{11}, h_{12}, \ldots, h_{nn}$ as parameters. 
Given the ideal $I_{\mathcal R_V} = \langle g_1, \ldots, g_\ell\rangle  \subseteq \mathbb C[\psi_1, \ldots,\psi_n, h_{11}, h_{12}, \ldots, h_{nn}]$ of the correspondence, a point in the critical locus is nonreduced or lies on a positive dimensional component precisely when the following $\ell \times n$ matrix drops rank: 
\begin{equation}\label{eqn:JpsiRV}
    {\rm Jac}_\psi(\mathcal R_V) = 
    \begin{pmatrix}
        \frac{\partial g_1}{\partial \psi_1} & \frac{\partial g_1}{\partial \psi_2} & \cdots & \frac{\partial g_1}{\partial \psi_n}\\
        \frac{\partial g_2}{\partial \psi_1} & \frac{\partial g_2}{\partial \psi_2} & \cdots & \frac{\partial g_2}{\partial \psi_n}\\
        \vdots  & \vdots & \ddots & \vdots \\
        \frac{\partial g_\ell}{\partial \psi_1} & \frac{\partial g_\ell}{\partial \psi_2} & \cdots & \frac{\partial g_\ell}{\partial \psi_n}\\
    \end{pmatrix}.
\end{equation}
The RR discriminant can therefore be written as
   \begin{equation*}
        \Sigma_V = 
        \pi_2(\{
        (\psi, H) \in \mathcal R_V : 
        {\rm rank}\,\, {\rm Jac}_\psi(\mathcal R_V)(\psi, H) \leq n-2
    \}).
    \end{equation*}

The RR discriminant can also be computed from the parametric RR correspondence.
Suppose that the parametrization space $U$ has (projective or affine) coordinates $x_1, \ldots, x_m$.
If $I_{\mathcal P_V} = \langle f_1, \ldots, f_p \rangle \subseteq \mathbb C[x_1, \ldots, x_m, h_{11}, h_{12}, \ldots, h_{nn}]$, we define the $p \times m$ Jacobian 
\begin{equation*}\label{eqn:JxPV}
    {\rm Jac}_x(\mathcal P_V) = 
    \begin{pmatrix}
        \frac{\partial f_1}{\partial x_1} & \frac{\partial f_1}{\partial x_2} & \cdots & \frac{\partial f_1}{\partial x_m}\\
        \frac{\partial f_2}{\partial x_1} & \frac{\partial f_2}{\partial x_2} & \cdots & \frac{\partial f_2}{\partial x_m}\\
        \vdots  & \vdots & \ddots & \vdots \\
        \frac{\partial f_p}{\partial x_1} & \frac{\partial f_p}{\partial x_2} & \cdots & \frac{\partial f_p}{\partial x_m}\\
    \end{pmatrix}.
\end{equation*}
\begin{proposition}\label{prop:param-dis}
    Suppose $V$ admits a birational parametrization $\psi: U \dashrightarrow V$. 
    Then the RR discriminant is
    \begin{equation*}
        \Sigma_V = 
        \pi_2(\{
        (x, H) \in \mathcal P_V : 
        {\rm rank}\,\, {\rm Jac}_x(\mathcal P_V)(x, H) \leq \dim U -1
    \}).
    \end{equation*}
\end{proposition}
\begin{proof}
    Let $\mathcal R_V^{\rm ram} = \{
        (\psi, H) \in \mathcal R_V : 
        {\rm rank}\,\, {\rm Jac}_\psi(\mathcal R_V) \leq n-2
    \}$  
    and 
    $\mathcal P_V^{\rm ram} = \{(x, H) \in \mathcal P_V :  {\rm rank}\,\, {\rm Jac}_x(\mathcal P_V) \leq \dim U - 1 \}$.
    It suffices to prove that the image of $\psi \times {\rm Id}: \mathcal P^{\rm ram}_V \dashrightarrow \mathcal R_V$ is 
    dense in 
    $\mathcal R^{\rm ram}_V$.
    In this case, $\psi \times {\rm Id}$ gives a birational map $\mathcal P_V^{\rm ram}  \dashrightarrow \mathcal R_V^{\rm ram}$ sending $$(x, H) \mapsto (\psi(x), H)$$ and so the projections onto the second factor of these sets will be equal.
    
    Suppose $(\hat x,\hat H)$ lies in the dense open set
$
\bigl(\psi^{-1}(V_{\mathrm{reg}})\times
\mathbb P(\operatorname{Sym}^2(\mathbb C^n))\bigr)\cap {\cal P}^{\mathrm{ram}}_V
$.
    Since $\psi(\hat x)$ is a regular point, the Jacobian ${\rm Jac}_x(\psi)(\hat x)$ has full rank equal to $\dim U$, so the map $${\rm Jac}_x(\psi)(\hat x): T_{\hat x} \gamma_2^{-1}(\hat H) \to T_{\psi(\hat x)} \pi_2^{-1}(\hat H) $$ between tangent spaces of the parametric and implicit critical loci of $\hat H$ is injective.
    Because ${\rm rank}\,\, {\rm Jac}_x(\mathcal P_V)(\hat x, \hat H) \leq \dim U - 1$, the tangent space $T_{\hat x} \gamma_2^{-1}(\hat H)$ has dimension at least $1$. 
    Therefore the image of $T_{\hat x} \gamma_2^{-1}(\hat H) $ under the injective map ${\rm Jac}_x(\psi)(\hat x)$ has dimension at least one. 
    Hence, the tangent space $T_{\psi(\hat x)} \pi_2^{-1}(\hat H) $ has dimension at least $1$, and therefore 
    ${\rm Jac}_\psi(\mathcal P_V)(\psi(\hat x), \hat H)$ has rank at most $n - 2$.
    This proves that $\psi \times {\rm Id}$ maps $\mathcal P^{\rm ram}_V$ into $\mathcal R^{\rm ram}_V$.
    By construction, the image $(V_{\rm reg} \times \mathbb P({\rm Sym}^2(\mathbb C^n)) \cap  \mathcal R^{\rm ram}_V$ of the dense open set under $\psi \times {\rm Id}$ is dense in $\mathcal R^{\rm ram}_V$.     
\end{proof}
We introduced $\Sigma_V$ as the set of matrices $H$ whose critical locus is infinite or singular.
This is the perspective taken in \cite{KKT25} when defining the discriminant of maximum likelihood estimation. 
The authors of \cite{DHGSR} take a different perspective, defining the discriminant of the Euclidean distance optimization problem as the branch locus of the projection $\pi_2$ onto the second factor. 
We now argue that these two definitions agree for RR discriminants. 
\begin{proposition}
    The RR discriminant of $V$ is equal to the branch locus of $$\pi_2: \mathcal R_V \to \mathbb P({\rm Sym}^2(\mathbb C^n)).$$    
\end{proposition}
\begin{proof}
    The branch locus consists of the matrices $H$ for which the differential ${\rm Jac}(\pi_2)$ of the projection $\pi_2$ is not surjective.
    We work in the affine chart where $\psi_1 = h_{11} = 1$.
    The Jacobian of $\mathcal R_V$ at a point $(\hat \psi, \hat H)$ with $\hat \psi \in V_{\rm reg}$ may be written as 
    $${\rm Jac}(\mathcal R_V)(\hat \psi, \hat H) = \begin{pmatrix} {\rm aJac}_\psi(\mathcal R_V)(\hat \psi, \hat H) &  {\rm aJac}_H(\mathcal R_V)(\hat \psi, \hat H)\end{pmatrix}$$
    where ${\rm aJac}_\psi(\mathcal R_V)$ consists of the last $n-1$ columns of \eqref{eqn:JpsiRV} and ${\rm aJac}_H(\mathcal R_V)$ is the analogous $\ell \times (\binom{n + 1}{2}-1)$ matrix where derivatives are taken with respect to $h_{12}, h_{13}, \ldots, h_{nn}$.
    There exists a $(n + \binom{n+1}{2}-2) \times (\binom{n+1}{2}-1)$ matrix $A = \begin{pmatrix} A_{\hat \psi}\\A_{\hat H} \end{pmatrix}$ whose columns form a basis for the kernel of ${\rm Jac}(\mathcal R_V)$, i.e., they form a basis for the tangent space of $\mathcal R_V$ at $(\hat \psi, \hat H)$. 
    The column span of $A_{\hat H}$ is the image of the tangent space $T_{(\hat \psi, \hat H)}\mathcal R_V$ under the differential map ${\rm Jac}(\pi_2)$.
    For general $\hat H$, this map is surjective and $A_{\hat H}$ has full rank.
    
    If $\hat H$ is in the branch locus, then $A_{\hat H}$ is not full rank, and
    we may choose $A_{\hat H}$ so that its last column is zero. 
    This choice forces the last column of $A_{\hat \psi}$ to be in the kernel of ${\rm aJac}_\psi(\mathcal R_V)(\hat \psi, \hat H)$. 
    Hence ${\rm aJac}_\psi(\mathcal R_V)(\hat \psi, \hat H)$ has rank less than $n - 1$ and so $\hat H$ is contained in $\Sigma_V$.
    Conversely, if $\hat H \in \Sigma_V$, then the kernel of ${\rm aJac}_\psi(\mathcal R_V)(\hat \psi, \hat H)$ contains a nonzero vector~$v$.
    Then the vector $\begin{pmatrix} v & 0 \end{pmatrix}^{\top}$ is in the column span of $A$. 
    Hence $A_{\hat H}$ drops rank and so $\hat H$ is in the branch locus of $\pi_2$.
    Thus $A_{\hat H}$ drops rank precisely when $\operatorname{Jac}_\psi(\mathcal R_V)(\hat \psi, \hat H)$ drops rank, and the two discriminants match. 
\end{proof}

We now examine the RR discriminants of projective spaces and  rational normal curves, before turning to the components of RR discriminants. 

\begin{proposition}
    The RR discriminant of $\mathbb P^{n-1}$ is the discriminant of the characteristic polynomial of a symmetric matrix. The defining polynomial has degree $n(n-1)$.
\end{proposition}
\begin{proof}
    The RR correspondence of $\mathbb P^{n-1}$ is 
    \begin{equation*}
        \mathcal R_{\mathbb P^{n-1}} = \overline{\left \{(\psi, H) \in \mathbb P^{n-1} \times \mathbb P({\rm Sym}^2(\mathbb C^n)): \psi \textrm{ is an eigenvector of } H \textrm{ and } \psi^T\psi \neq 0 \right\}}.
    \end{equation*}
    It follows that the fiber of a complex symmetric matrix $H$ is 
    \begin{align*}
        \pi_2^{-1}(H) = \overline{ \{\psi \in \mathbb P^{n-1} : \psi \textrm{ is an eigenvector of } H \textrm{ and } \psi^T\psi \neq 0 \}}.
    \end{align*}
    We seek conditions on $H$ under which this set is singular. 
    If the characteristic polynomial of a matrix $\hat H$ has no double roots, then $\hat H$ is diagonalizable and has $n$ distinct eigenvectors. 
    Thus $\pi_2^{-1}(\hat H)$ consists of $n$ isolated points if the characteristic polynomial of $H$ has no double root.
    Therefore the discriminant of the characteristic polynomial of $\hat H$ contains $\Sigma_{\mathbb P^{n-1}}$ 
    
    Conversely, suppose $\hat H$ has a repeated eigenvalue $\lambda$.
    If the corresponding eigenspace has dimension greater than $1$, then $\pi_2^{-1}(\hat H)$ has a positive dimensional component, so $\hat H$ is in $\Sigma_{\mathbb P^{n-1}}$. 
    If $\lambda$ has geometric multiplicity 1, then the corresponding eigenspace is spanned by a single vector $\hat \psi$. 
    Consider a curve $H: [0,1] \to \mathbb P({\rm Sym^2(\mathbb C^n}))$ such that $\hat H = H(1)$ and $H(t)$ is diagonalizable for $t \in [0,1)$.
    The eigenvalues and eigenvectors of $H(t)$ vary continuously, so there exist functions $\lambda_1, \lambda_2: [0,1] \to \mathbb C$ and $\psi_1,  \psi_2: [0,1] \to \mathbb P^{n-1}$ 
    such that 
    \begin{itemize}
        \item $H(t)\psi_1(t) = \lambda_1(t) \psi_1(t)$ and  $H(t)\psi_2(t) = \lambda_2(t)\psi_2(t)$ for all $t \in [0,1]$,
        \item $\lambda_1(t) \neq \lambda_2(t)$ and $\psi_1(t) \neq \psi_2(t)$ for $t \in [0,1)$, and 
        \item  $\lambda_1(1) = \lambda_2(1) = \lambda$.
    \end{itemize}
    At $t = 1$, we have $\hat H\psi_1(1) = \lambda\psi_1(1)$ and $\hat H\psi_2(1) = \lambda\psi_2(1)$. 
    Since the eigenspace of $\lambda$ is one-dimensional, $\psi_1(1) = \psi_2(1) = \hat \psi$ and hence $\hat \psi$ is a nonreduced point of $\pi^{-1}(\hat H)$. 
    Therefore, if the characteristic polynomial of $\hat H$ has a double root, then $\hat H \in \Sigma_V$.
\end{proof}

\begin{remark}
    If $\hat H$ is not contained in $\Sigma_{\mathbb P^{n-1}}$ then no points in fiber $\pi_2^{-1}(H)$ lie on the isotropic quadric $Q$.
    This is because a symmetric matrix has an isotropic eigenvector if and only if its characteristic polynomial has a repeated eigenvalue  \cite{scott1993}.
    Consequently, for every $\hat H \notin \Sigma_V$, 
    \eqref{eqn:optproblem} has exactly RR degree many critical points.
\end{remark}

\begin{proposition}
    The RR degree of the rational normal curve $C \subseteq \mathbb P^{d}$ of degree $d$ is $2(2d - 1)$. 
    The parametric RR correspondence of the rational normal curve is a hypersurface in $\mathbb P^1 \times \mathbb P({\rm Sym}^2(\mathbb C^{d+1}))$ which is linear in $h_{11}, h_{12}, \ldots, h_{d+1,d+1}$.
    The RR discriminant is a hypersurface of degree $2(4d - 3)$.
\end{proposition}
\begin{proof}
    The first claim follows from \cite[Proposition 3.6]{SSW25}. For the RR correspondence we use the usual parametrization from  $\mathbb{P}^1$. 
    Since the dimension of the parametric RR correspondence is $\binom{d+2}{2}-1$, it is a hypersurface in $\mathbb P^1 \times \mathbb P({\rm Sym}^2(\mathbb C^{d+1}))$. For each $ p\in \mathbb{P}^1$, 
    the fiber $\gamma_1^{-1}(p) \subseteq \mathcal P_{C}$ is a linear space. 
    Therefore the multidegree of this hypersurface is $(2(2d-1),1)$. 
    By Proposition~\ref{prop:param-dis}, the RR discriminant is the discriminant of this binary form of degree $2(2d-1)$ whose coefficients are linear in $H$. Therefore its degree is $2(2(2d-1))-2 = 2(4d-3)$.
\end{proof}

We now turn to the structure of the discriminant. 
The RR discriminant is generally not irreducible.
Let $\mathcal R^{\rm ram}_V = \{
        (\psi, H) \in \mathcal R_V : 
        {\rm rank}\,\, J_\psi(\mathcal R_V) \leq n-2\}$
denote the ramification locus of the projection $\pi_2$.         
We can write $\mathcal R^{\rm ram}_V$ as the union 
\begin{equation*}
    \mathcal R^{\rm ram}_V = 
    (\mathcal R^{\rm ram}_V \cap Q) \cup 
    \overline{(\mathcal R^{\rm ram}_V \setminus Q)}
\end{equation*}
where $Q$ denotes $\mathcal V(\psi^T\psi) \subseteq \mathbb P^{n-1} \times \mathbb P({\rm Sym}(\mathbb C^n))$.
The projections of these components $\pi_2(\mathcal R^{\rm ram}_V \cap Q)$ and $\pi_2(\overline{(\mathcal R^{\rm ram}_V \setminus Q)})$ are often, but not always, hypersurfaces.
We call the set $\pi_2(\mathcal R^{\rm ram}_V \cap Q)$ the {\em isotropic part} of the discriminant and we call $\pi_2(\overline{(\mathcal R^{\rm ram}_V \setminus Q)})$ the {\em nonisotropic part} of the discriminant. 
We illustrate this phenomenon with some small examples.

\begin{example}\label{ex:Pn-dis}
    The nonisotropic part of $\Sigma_{\mathbb P^{n-1}}$ is a proper subvariety of $\Sigma_{\mathbb P^{n-1}}$ with codimension greater than 1 and describes matrices with diagonal blocks. 
    The isotropic part of $\Sigma_{\mathbb P^{n-1}}$ is equal to $\Sigma_{\mathbb P^{n-1}}$.
    Its defining polynomial is a sum of squares; see  \cite{ilyushechkin}. 
    For example, $\Sigma_{\mathbb P^1}$ is the hypersurface cut out by $(h_{11} - h_{22})^2 + 4h_{12}^2$. 
    The nonisotropic part $\mathcal V(h_{12}, h_{11} - h_{22})$ of $\Sigma_{\mathbb P^1}$ has codimension 2 and consists of scalar multiples of the $2\times 2$ identity matrix.
\end{example}

\begin{example}\label{ex:quadric-dis}
    The rational normal quadric has RR degree 6. 
    The RR discriminant is a degree-10 hypersurface.
    Its defining polynomial has three irreducible factors over $\mathbb{Q}$. The two isotropic factors are
    \begin{align*}
&3(2h_{12}-2h_{13} - h_{22} + h_{33})^2 + (2h_{11}-2h_{12}-2h_{13}-  h_{22} + 4 h_{23} - h_{33})^2  \textrm{ and }\\
&3(2h_{12}+ 2h_{13} + h_{22} - h_{33})^2 + (2h_{11}+2h_{12}-2h_{13}-  h_{22} - 4 h_{23} - h_{33})^2.
\end{align*}
The nonisotropic factor (see \cite{zenodo_our}) has degree $6$ and divides $\mathbb P({\rm Sym}^2(\mathbb R^3))$ into two regions. 
We computed the regions in the complement of the nonisotropic factor using the \verb+Julia+ software \verb+HypersurfaceRegions.jl+ \cite{BSW25}. 
In one region, \eqref{eqn:optproblem} has 2 real critical points; in the other it has 4 real critical points. 
For example, \eqref{eqn:optproblem}
has 2 real critical points for the left matrix and 4 real critical points for the right matrix:
\begin{align*}
    H = \begin{footnotesize}
        \begin{pmatrix}
        1 & 3 & 5\\ 3 & 7 & 9\\ 5 & 9 & 11
    \end{pmatrix}, \end{footnotesize}
    \qquad
    H =  \begin{footnotesize} \begin{pmatrix}
        8 & 1 & 13 \\ 1 & 11 & 7 \\ 13 & 7 & 5
    \end{pmatrix} \end{footnotesize}.
\end{align*}
This discriminant also appears in \cite[Example 5.8]{SSW25}.
\end{example}

In the following examples, we were not able to compute the full discriminant polynomials. 
The degrees were computed symbolically by intersecting the discriminant hypersurface with a line. 
We can give a lower bound on the number of regions in the discriminant complement by computing all critical points for \eqref{eqn:opt-param} for random real matrices and counting the numbers of real solutions that arise. 

\begin{example}\label{ex:cubic-dis}
     The twisted cubic has RR degree 10. 
     The RR discriminant has degree 18
     and its defining polynomial factors into three components over $\mathbb Q$. 
    The first isotropic factor is a sum of squares of rank 2: 
    \begin{equation*}
        (h_{11} - 2h_{13} - h_{22} + 2h_{24} + h_{33} - h_{44} )^2 
        + 
        4(h_{12} - h_{14} - h_{23} + h_{34})^2.
    \end{equation*}
    The second isotropic factor (see \cite{zenodo_our}) is a quartic and a sum of squares of rank at most 4. 
    We checked this using the \verb+Julia+ packages \verb+COSMO.jl+ \cite{cosmo}, \verb+JuMP.jl+ \cite{jump}, and \verb+SumOfSquares.jl+~\cite{sos}.
    The nonisotropic factor has degree 12 and divides $\mathbb P({\rm Sym}^2(\mathbb R^4))$ into at least three regions, exhibited by the fact that \eqref{eqn:optproblem} can have 2, 4, or 6 real critical points. 
\end{example}

\begin{conjecture}~\label{conj:rnc-disc-deg}
    The nonisotropic part of the RR discriminant of the rational normal curve of degree $d \geq 2$ is a hypersurface of degree $6(d-1)$.
    The isotropic part has degree $2d$.
\end{conjecture}

We have verified Conjecture~\ref{conj:rnc-disc-deg} numerically for $d \leq 10$; see \cite{zenodo_our} for computations. 

\begin{example}
    The RR degree of $\mathbb P^1 \times \mathbb P^1$ is 8.
    The RR discriminant of $\mathbb P^1 \times \mathbb P^1$ is a hypersurface of degree 32.
    The isotropic part has degree 8.
    The nonisotropic part has degree 24. 
    The discriminant divides $\mathbb P({\rm Sym}^2(\mathbb R^4))$ into at least three regions: \eqref{eqn:optproblem} can have 4, 6, or 8 real critical points for $\mathbb{P}^1 \times \mathbb{P}^1$. 
\end{example}

\begin{example}
    Consider the Hirzebruch surface parametrized by $\psi: \mathbb C^2 \to \mathbb P^4$ defined by $(a, b) \mapsto [1:a:b:ab:a^2b]$.
    The RR degree of this variety is 20.
    The RR correspondence is a $14$-dimensional subvariety of $\mathbb{P}^4 \times \mathbb P^{14}$. 
    The RR discriminant is a hypersurface of degree 74.
    Its isotropic part has degree 26 and its nonisotropic part has degree 48.
    The discriminant divides $\mathbb P({\rm Sym}^2(\mathbb R^5))$ into at least 5 regions: 
    the optimization problem \eqref{eqn:optproblem} can have 4, 6, 8, 10, or 12 real critical points.
\end{example}

It is not clear whether there can exist $(\psi, H) \in \mathcal R^{\rm ram}_V \cap Q$ with real $H$. 
The vector $\psi$ cannot be real, since $\psi^T\psi = 0$ is a sum of squares. 
In the examples we have computed, the variety $\pi_2(\mathcal R^{\rm ram}_V \cap Q)$ has no real points and its defining polynomial factors as a sum of squares over~$\mathbb Q$. 
We therefore make the following conjecture, which we have verified for the rational quadric (Example~\ref{ex:quadric-dis}), the twisted cubic (Example~\ref{ex:cubic-dis}), and projective space (Example~\ref{ex:Pn-dis}). 
\begin{conjecture}\label{conj:sos}
    Every irreducible isotropic factor of the defining polynomial of the RR discriminant is a sum of squares. 
\end{conjecture}

    One motivation for studying the RR discriminant of a variety 
    is to compute the average RR degree, that is, the expected number of real critical points. 
    The role of the discriminant in this computation is explained in the Euclidean distance case in \cite{DHGSR}. 
    In particular, in this application we are only interested in the real locus of the discriminant. 
    Thus, if Conjecture~\ref{conj:sos} holds, then only the nonisotropic part of the discriminant is relevant.
    We therefore present degrees of the nonisotropic parts of RR discriminants of various varieties in Tables~\ref{tab:RRdegtensors} and \ref{tab:RR-disc-degs}, and make the following conjecture.
    \begin{conjecture}
        The nonisotropic part of the RR discriminant of $\mathbb P^1 \times \mathbb P^{n-1}$ is a hypersurface of degree $24 \binom{n+1}{3}$.
        The nonisotropic part of the RR discriminant of $\mathbb P^2 \times \mathbb P^{n-1}$ is a hypersurface of degree $24n^2 \binom{n}{2}$.
    \end{conjecture}

\section{Bombieri-Weyl Correspondence}\label{sec:BW}
In this section, we will introduce yet another correspondence that essentially parametrizes pairs $(\psi, H) \in V \times \mathrm{Sym}^2(\mathbb{C}^n)$ where $\psi$ is a critical point of the Rayleigh-Ritz optimization problem \eqref{eqn:optproblem}. 
For this, we will consider $V$ in its second Veronese embedding $X = \nu_2(V)$ and transform the Rayleigh-Ritz optimization problem \eqref{eqn:optproblem} to a distance optimization problem using {\it Bombieri-Weyl} distance. 
The equivalence of these two optimization problems has been introduced and fully developed in \cite[Section 2]{SSW25}. We give a very brief introduction to the Bombieri-Weyl distance optimization and cite the main result we will use. 
This optimization problem is a variant of the Euclidean distance optimization problem.
We will show that, as in the Euclidean distance case, the \emph{Bombieri-Weyl correspondence} (BW correspondence) of a parametrized variety is itself parametrized. We include this with the hope that it will be useful for computing the \emph{average} number of \emph{real} critical points of \eqref{eqn:optproblem}; see \cite[Section 4]{DHGSR}.

First, note that the Rayleigh-Ritz optimization problem \eqref{eqn:optproblem} is equivalent to the following optimization problem
\begin{align}
    \label{eqn:optsphere}
    &\textrm{minimize} \,\, f_H(\psi) = \psi^T H \psi \quad \, 
    \textrm{subject to} \,\, \psi \in C(V_\mathbb{R}) \cap S^{n-1}
\end{align}
where $C(V_\mathbb{R}) \subset \mathbb{R}^n$ is the cone over the projective variety $V \subset \mathbb{P}^{n-1}$ and $S^{n-1}$ is the $(n-1)$-dimensional sphere in $\mathbb{R}^n$. We also point out that each critical point $\hat{\psi}$ of \eqref{eqn:optproblem} corresponds to two critical points $\pm \frac{1}{(\hat{\psi})^T \hat{\psi}} \hat{\psi}$ of \eqref{eqn:optsphere}.

Recall that any homogeneous form $f$ of degree $d$ in $n$ variables with real coefficients can be represented by an element of  $\Sym^d(\R^n)$. This representation as a symmetric tensor can be written as 
$$ f(x) = \sum_{|\alpha| = d} \binom{d}{\alpha}f_\alpha x^\alpha,$$
where $\binom{d}{\alpha}$ is the multinomial coefficient $\frac{d!}{\alpha_1 ! \cdots \alpha_n !}$. Hence $f$ is identified with $(f_\alpha)_{|\alpha| = d}$. 
The case we will work with 
is $d=2$. A homogeneous quadratic polynomial in $n$ variables and the corresponding symmetric matrix of this quadratic form are
\begin{equation} \label{eq:quadratic}
\sum_{i=1}^n h_{ii} x_i^2 + \sum_{i<j} 2 h_{ij}x_ix_j,  \quad \quad  \quad
H =  \begin{footnotesize} \begin{pmatrix}
h_{11} &  h_{12} & \cdots &  h_{1n} \\
h_{12} & h_{22} & \cdots & h_{2n} \\
\vdots &    &  \ddots & \vdots \\
h_{1n} & h_{2n} & \cdots & h_{nn}
\end{pmatrix} \end{footnotesize}.
\end{equation}
From now on we will identify a symmetric matrix $H$ with the 
quadratic form $f_H = x^T H x$.
\begin{definition}
The Bombieri-Weyl inner product between two 
homogeneous forms $f$ and $g$ in $\R[x_1, \ldots, x_n]_d$ is 
$$\langle f, g\rangle_{\mathrm{BW}} := \sum_{|\alpha|=d} \binom{d}{\alpha} f_\alpha g_\alpha.$$
The corresponding Bombieri-Weyl norm of $f$
is the square root of 
$$ q_{\mathrm{BW}}(f) := \sum_{|\alpha|=d} \binom{d}{\alpha} f_\alpha^2.$$ 
If $f$ is a quadratic polynomial as in \eqref{eq:quadratic} then
 $q_{\mathrm{BW}}(f) = \mathrm{trace}(H^2)$.
\end{definition}
\begin{proposition} \cite[Proposition 2.6]{SSW25} \label{prop:SSW25} 
Let $V \subset \mathbb{P}^{n-1}$ be an irreducible variety and let $X = \nu_2(V)$. A point $\lambda \nu_2(\psi) \in C(X)$ with 
$\psi \in S^{n-1}$ and $\lambda \in \mathbb{C}$ is a critical point of 
 \begin{align}
    \label{eqn:optBW}
    &\textrm{minimize} \,\, q_{\mathrm{BW}}(f_H- g)  \quad \, 
    \textrm{subject to} \,\, g \in C(X) 
\end{align}  
if and only if $\psi$ is a critical point of \eqref{eqn:optsphere} with 
$\lambda = f_H(\psi) = \psi^T H\psi$.
\end{proposition}
\begin{proposition} \label{prop:BW-crit}
A point $f_Z$ in $X$ is a critical point of \eqref{eqn:optBW} if and only if   
$Z-H \in T_Z(X)^{\perp}$ where $\perp$ refers to the $\mathrm{BW}$ inner product. 
\end{proposition}
\begin{proof}
Each point $f_Z$ on $X$ is 
of the form $\psi^TZ\psi$ for a symmetric matrix $Z$ of rank one of the form \eqref{eq:quadratic}. Then the gradient of $q_{\mathrm{BW}}(f_H - f_Z) = \mathrm{trace}((H-Z)^2)$ with respect to $Z$ is equal to $-2(H-Z)$. 
Now the result follows. 
\end{proof}
Let $I_X = \langle g_1, \ldots, g_s\rangle$ be the prime ideal in $\mathbb{R}[Z] = \mathbb{R}[z_{ij} \, : \, 1\leq i \leq j \leq n]$ defining $X$. Suppose 
the codimension of $I_X$ is $c$. We let $\mathrm{Jac}(I_X)$ to be the Jacobian of the ideal $I_X$. Then the critical ideal 
is 
\begin{equation} 
\label{eqn:critical-ideal}
\left( I_X + \bigg \langle (c+1) \mbox{-minors of } \left( \begin{array}{c}
\nabla_Z(\mathrm{trace}((Z-H)^2)) \\ \mathrm{Jac}(I_X) \end{array} \right)\bigg  \rangle \right) \, : \, I_{X_{\rm sing}}^\infty.     
\end{equation}

\begin{definition}
The BW correspondence $\mathcal B_V$ is 
\begin{equation}\label{eqn:BW-cor}
    \mathcal B_V = \overline{\{(Z, H) \in C(X)_{\rm reg} \times \mathbb  \Sym^2(\mathbb C^n): 
    Z \textrm{ is a critical point of (\ref{eqn:optBW}})
    \}}.
\end{equation}
where $C(X)$ is the cone over $X = \nu_2(V) \subset \mathbb{P}(\Sym^2(\mathbb{C}^n))$
for the irreducible variety~$V \subset \mathbb{P}^{n-1}$.
\end{definition}
\begin{theorem}
The BW correspondence $\mathcal B_V$ where $X = \nu_2(V)$ is irreducible of codimension $c$ is an irreducible variety of dimension $\binom{n+1}{2}$ in
$\Sym^2(\C^n) \times \Sym^2(\C^n)$. The first projection
$\pi_1 \, : \, \mathcal B_V \rightarrow C(X)$ is an affine vector bundle over $
C(X)_{\mathrm{reg}} = C(X) \setminus C(X)_{\rm sing}$ of rank $c$. Over a general symmetric matrix $H \in \Sym^2(\C^n)$, the second projection $\pi_2 \, : \, \mathcal B_V \rightarrow \Sym^2(\C^n)$ has finite fibers $\pi_2^{-1}(H)$ with cardinality equal to two times the RR degree of $V$. 
\end{theorem}
\begin{proof}
 For a fixed $Z \in C(X)_{\mathrm{reg}}$, the fiber $\pi_1^{-1}(Z)$ is equal to $\{Z\} \times \{Z + T_Z(X)^{\perp}\}$  by 
Proposition \ref{prop:BW-crit}. 
Here, the second factor is an affine
 space of dimension $c$. 
This shows that $\mathcal B_V$ is an affine vector bundle of 
dimension $\binom{n+1}{2}$. Since $\pi_2$ is a dominant map, over a general point $H$, the cardinality of
$\pi_2^{-1}(H)$ is finite and is equal to the degree of the map $\pi_2$. Now, Proposition~\ref{prop:SSW25} implies that this is equal to twice the RR degree of $V$. 
\end{proof} 
\begin{corollary} For general $H \in \Sym^2(\C^n)$, the critical ideal \eqref{eqn:critical-ideal} is zero-dimensional, i.e., the variety defined by the critical ideal is finite in $\Sym^2(\C^n)$.
\end{corollary}
\begin{proof}
For a fixed $H$, the zero set of the critical ideal \eqref{eqn:critical-ideal} is $\pi_2^{-1}(H)$.
\end{proof}

\begin{corollary} \label{cor:rational-BW}
If $V$ is rational so is the BW correspondence $\mathcal B_V$. 
\end{corollary}
\begin{proof}
This follows from \cite[Corollary 4.2]{DHGSR}, but we include it for completeness. 
If $V$ is rational, then $X = \nu_2(V)$, and hence $C(X)$ also
has a rational parametrization $\psi \, : \, \C^m \dashrightarrow \Sym^2(\C^n)$ where $m = \dim(C(X)) = \binom{n+1}{2} - c$. The Jacobian $\mathrm{Jac}(\psi)$ of this map
is an $\binom{n+1}{2} \times m$ matrix of rational functions in the coordinates $t_1, \ldots, t_m$ of $\C^m$. The columns of the Jacobian span the tangent space of $C(X)$ at the point $\psi(t)$. The left kernel
of $\mathrm{Jac}(\psi)$ is a linear space of dimension $c$. By using Cramer's rule one can get a basis 
$\{\beta_1(t), \ldots, \beta_c(t) \}$ for this linear space where $\beta_j(t)$ are also rational functions. With this, the map
$$ \C^m \times \C^c \dashrightarrow \mathcal B_V, \quad \quad
(t,s) \mapsto \left(\psi(t), \, \psi(t) + \sum_{j=1}^c s_j \beta_j(t)\right)$$
is a rational parametrization of $\mathcal B_V$.
\end{proof}

\section{Geometry of Tensor Train Varieties} \label{sec: TTvarieties}
In this section, we introduce \emph{tensor train varieties} (TT varieties) -- projective varieties of tensors in a tensor train format. These varieties  can be described by the combinatorial structure of
a graph following the tensor network viewpoint as, for example,  in \cite{  Landsberg2012, Seynnaeve2025,Szalay2015}.

Let $\Gamma_n$ be an undirected path graph on vertices $\{0,1,\dots,n\}$, with edges $\{i-1,i\}$ for $i=1,\dots,n$. We assign positive integers ${\mathbf k} = (k_0, k_1, \ldots, k_n)$ to the vertices, known as \emph{physical dimensions}, and ${\mathbf r} = (r_1, r_2, \ldots, r_n)$ to the edges, called \emph{bond dimensions}; by convention, we set $r_0=r_{n+1}=1$. We identify tensors in the ambient projective space
$\mathbb{P}\big(\bigotimes_{i=0}^n \mathbb{C}^{k_i}\big)$ with points having homogeneous coordinates $\psi_{j_0,\dots,j_n}$, where $j_i\in[k_i]:=\{1,\dots,k_i\}$.
For each index $i\in\{0,1,\ldots,n\}$ and each $j_i\in[k_i]$, let $A_i^{(j_i)}$ be a general $r_i\times r_{i+1}$ matrix of parameters.
\begin{definition}\label{def: TTvariety}
  The \emph{tensor train variety} $V_{\mathbf{k}, \mathbf{r}}$ is the projective variety in $\mathbb{P}\left( \bigotimes_{i=0}^n \mathbb{C}^{k_i} \right)$ defined as the Zariski closure of the image of the following rational map:
\begin{equation}\label{eq: infinite_param}
    \begin{aligned}
        \Psi_{\mathbf{k},\mathbf{r}}:
\prod_{i=0}^n\Big(\mathbb{P}\big(\mathbb{C}^{r_i\times r_{i+1}}\big)\Big)^{k_i}
& \dashrightarrow
\mathbb{P}\!\left(\bigotimes_{i=0}^n \mathbb{C}^{k_i}\right), \\
\big(\,A_i^{(j_i)}\,\big)_{0\le i\le n,\;1\le j_i\le k_i}
& \longmapsto
\big(\psi_{j_0,j_1,\dots,j_n}\big)_{(j_0,j_1,\dots,j_n)\in\prod_{i=0}^n[k_i]}\, ,
    \end{aligned}
\end{equation}
where $\psi_{j_0,j_1,\dots,j_n}=A_0^{(j_0)}A_1^{(j_1)}\cdots A_n^{(j_n)}$.  In symbols, we have  $V_{\mathbf{k},\mathbf{r}} = \overline{\operatorname{im}(\Psi_{\mathbf{k},\mathbf{r}})}$.
\end{definition}
It is often convenient to characterize $V_{\mathbf{k},\mathbf{r}}$ by rank constraints on certain matrix flattenings of a tensor. For $i \in \{0,1,\dots,n-1\}$, let $\psi^i$ denote the flattening of $\psi$ according to the partition
$\{0,1,\dots,i\}\sqcup\{i+1,\dots,n\}$, viewed as a $(k_0\cdots k_i)\times (k_{i+1}\cdots k_n)$ matrix.
 \begin{lemma}\label{lem: another definition of TT}
 The tensor train variety $V_{\mathbf{k}, \mathbf{r}}$ is equal to
 \begin{equation}\label{eq: another definition of TT}
   V_{\mathbf{k}, \mathbf{r}} = \left \{ \psi \in \mathbb{P}\left( \bigotimes_{i=0}^n \mathbb{C}^{k_i} \right) \colon \mathrm{rank}(\psi^i) \leq r_{i+1} \quad \text{for } i=0, \ldots, n-1 \right\}.
 \end{equation}
In particular, set-theoretically $V_{\mathbf{k}, \mathbf{r}}$ is cut by all $(r_{i+1} + 1)\textrm{-minors of flattenings $\psi^i$}$.  
 \end{lemma}
 \begin{proof}
     Let $\psi$ be in the image of $\Psi_{\mathbf{k},\mathbf{r}}$. Fix $i\in\{0,\dots,n-1\}$ and consider two matrices
\[
L_i:=A_0^{(j_0)}\cdots A_i^{(j_i)}\in\mathbb{C}^{1\times r_{i+1}},
\qquad
R_i:=A_{i+1}^{(j_{i+1})}\cdots A_n^{(j_n)}\in\mathbb{C}^{r_{i+1}\times 1}.
\]
Then the flattening $\psi^i$ has entries $\psi^i[j_0,\dots,j_i\, ;\, j_{i+1},\dots,j_n]
= L_i\,R_i$.
Thus the flattening $\psi^i$ factors through a vector space of dimension $r_{i+1}$, and therefore $\operatorname{rank}(\psi^i)\le r_{i+1}$.

Conversely, let $\operatorname{rank}(\psi^i)\le r_{i+1}$ for all $i=0,\dots,n-1$. We construct a TT decomposition inductively. 
Start with  $\psi^0\in\mathbb{C}^{k_0\times (k_1\cdots k_n)}$ and 
let $\rho_1 = \operatorname{rank}(\psi^0) \leq r_1$. Choose a rank factorization $\psi^0 = L_0 R_0$ with $L_0 \in \mathbb{C}^{k_0 \times \rho_1}$ of full column rank and $R_0 \in \mathbb{C}^{\rho_1 \times (k_1 \cdots k_n)}$ of full row rank. Define a $1 \times \rho_1$ matrix $A_0^{(j_0)}$ to be the $j_0$-th row of $L_0$, and reshape $R_0$ as a tensor 
$$
\psi^{(1)} \in \mathbb{C}^{\rho_1} \otimes \mathbb{C}^{k_1} \otimes \cdots \otimes \mathbb{C}^{k_n}.$$
Now consider the flattening of $\psi^{(1)}$ with respect to the partition $\{\rho_1, k_1\} \sqcup \{k_2, \dots, k_n\}$, which is a $(\rho_1 k_1) \times (k_2 \cdots k_n)$ matrix. Since $L_0$ has full  rank, so does the matrix $L_0 \otimes I_{k_1}$, and since
$$
\psi^1  = (L_0 \otimes I_{k_1}) \cdot \operatorname{flat}(\psi^{(1)}),
$$
we have $\rho_2 \coloneqq \operatorname{rank}(\operatorname{flat}(\psi^{(1)})) \leq r_2$.
 Again choose a rank factorization ${\operatorname{flat}(\psi^{(1)}) = L_1 R_1}$ with $L_1 \in \mathbb{C}^{(\rho_1 k_1) \times \rho_2}$ of full column rank and read the $k_1$ slices of $L_1$ as matrices $A_1^{(j_1)} \in \mathbb{C}^{\rho_1 \times \rho_2}$. Then, reshape $R_1$ to be the tensor $\psi^{(2)} \in \mathbb{C}^{\rho_2} \otimes \mathbb{C}^{k_2} \otimes \cdots \otimes \mathbb{C}^{k_n}$.
We proceed
inductively and at step $i=n$ (with $\rho_{n+1}=1$), the remainder $\psi^{(n+1)}$ is scalar, so we obtain
$$
\psi_{j_0, \dots, j_n} = A_0^{(j_0)} A_1^{(j_1)} \cdots A_n^{(j_n)} \quad 
\text{with } \; A_i^{(j_i)} \in \mathbb{C}^{\rho_i \times \rho_{i+1}} \; \text{ and } \; \forall \, i: \rho_i \leq r_i.$$ 
Extending each core with zero rows and columns to size $r_i \times r_{i+1}$ shows that $\psi \in \operatorname{im}(\Psi_{\mathbf{k},\mathbf{r}})$.
Finally, rank constraints are given by vanishing minors, so righthand side of \eqref{eq: another definition of  TT} is closed.
 \end{proof}
 \begin{corollary}
Let ${\bf k}=(k_0,\ldots,k_n)$ and ${\bf r}=(r_1,\ldots,r_n)$, with the convention
$r_0=r_{n+1}=1$. There exists a tensor $\psi\in V_{{\bf k},{\bf r}}$ whose flattening
ranks are exactly
$
\operatorname{rank}(\psi_{i-1})=r_i$ for  $i\in [n]
$
if and only if
\begin{equation}\label{eq: kr constraints}
r_i \leq r_{i-1}k_{i-1}
\quad \text{and} \quad
r_i \leq r_{i+1}k_i
\quad \text{for all } i=1,\ldots,n .    
\end{equation}
\end{corollary}
\begin{proof}
The necessity follows from the elementary rank inequalities
$$
\operatorname{rank}(\psi_{i-1})
   \leq k_{i-1}\operatorname{rank}(\psi_{i-2})
\quad\text{and}\quad
\operatorname{rank}(\psi_{i-1})
   \leq k_i\operatorname{rank}(\psi_i),
$$
with the conventions $r_0=r_{n+1}=1$. Thus, if
$\operatorname{rank}(\psi_{i-1})=r_i$ for all $i\in[n]$, then~\eqref{eq: kr constraints} holds.

Conversely, assume that inequalities in~\eqref{eq: kr constraints} hold. For a generic choice of parameters, the flattening
$\psi_{i-1}$ has rank
$
\min\{r_i,\, r_{i-1}k_{i-1},\, r_{i+1}k_i\}
$.
By the assumed inequalities this minimum is equal to $r_i$. Hence there
exists a tensor $\psi\in V_{{\bf k},{\bf r}}$ with
$\operatorname{rank}(\psi_{i-1})=r_i$ for all $i=1,\ldots,n$.
\end{proof}
\begin{remark}\label{rem: min bond dims}
If either of the inequalities in~\eqref{eq: kr constraints} for $r_i$ is an equality or fails to hold, then the corresponding rank condition in~\eqref{eq: another definition of TT} is automatically satisfied, as in the example below. In our computations, we
therefore replace $r_i$ by
$
\min\{r_i,\, r_{i-1}k_{i-1},\, r_{i+1}k_i\}
$
whenever necessary.
\end{remark}
\begin{example}\label{ex:2222}
Take $\mathbf{k}=(2,2,2,2)$ and $\mathbf{r}=(1,2,1)$.  The parametrization \eqref{eq: infinite_param} yields
\[
\psi_{ijkl}=A_0^{(i)}A_1^{(j)}A_2^{(k)}A_3^{(\ell)},
\qquad i,j,k,\ell\in\{1,2\},
\]
  where $A_0^{(i)}$ and $A_3^{(\ell)}$ are scalars, $A_1^{(j)}\in\mathbb{C}^{1\times 2}$, and $A_2^{(k)}\in\mathbb{C}^{2\times 1}$. Writing
\[
A_0^{(i)}=(a_0^{i}),\quad
A_1^{(j)}= \begin{pmatrix} a_{11}^{j} & a_{12}^{j}\end{pmatrix},\quad
A_2^{(k)}=\begin{pmatrix} a_{21}^{k} \\ a_{22}^{k}\end{pmatrix},\quad
A_3^{(\ell)}=(a_3^{\ell}),
\]
one checks that $V_{\mathbf{k},\mathbf{r}}$ is isomorphic to the Segre product $\mathbb{P}^3\times (\mathbb{P}^1)^2$ whose vanishing ideal is minimally generated by the $2$-minors of the flattenings $\psi^0$ and $\psi^2$. Minors of $\psi^1$ don't appear since $r_2 = 2 = 1\cdot2 = r_{1}k_{1}$.
The corresponding tensor network is shown in Figure~\ref{fig: expathgraph}.
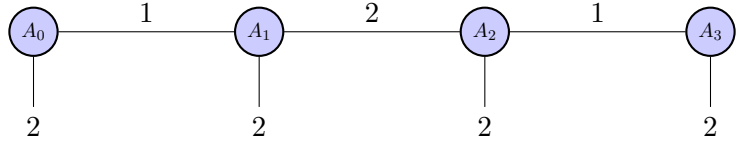
\begin{figure}
    \centering
    \begin{tikzpicture}[
    roundnode/.style={circle, draw=black, fill=blue!20, thick, minimum size=5mm, scale = 0.7},
]
\node[roundnode] (A1) at (0,0) {$A_0$};
\node[roundnode] (A2) at (3,0) {$A_1$};
\node[roundnode] (A3) at (6,0) {$A_2$};
\node[roundnode] (A4) at (9,0) {$A_3$};
\draw[-] (A1) -- (A2) node[midway, above] {$1$};
\draw[-] (A2) -- (A3) node[midway, above] {$2$};
\draw[-] (A3) -- (A4) node[midway, above] {$1$};
\draw[-] (A1) -- ++(0,-1) node[below] {$2$};
\draw[-] (A2) -- ++(0,-1) node[below] {$2$};
\draw[-] (A3) -- ++(0,-1) node[below] {$2$};
\draw[-] (A4) -- ++(0,-1) node[below] {$2$};
\end{tikzpicture}
    \caption{A tensor network diagram for binary tensors of order 4 of TT rank ${\bf r} = (1,2,1)$.}
    \label{fig: expathgraph}
\end{figure}
\end{example}
\begin{lemma}
\label{lem: bond_one_segre}
Suppose that $r_i=1$ for some $i\in[n]$. Let
$
{\bf k}'=(k_0,\ldots,k_{i-1}), {\bf r}'=(r_1,\ldots,r_{i-1}),
$
and
$
{\bf k}''=(k_i,\ldots,k_n), {\bf r}''=(r_{i+1},\ldots,r_n).
$
Then $V_{{\bf k},{\bf r}}$ is the Segre embedding of
$
V_{{\bf k}',{\bf r}'}\times V_{{\bf k}'',{\bf r}''}.
$
\end{lemma}
\begin{proof}
Let $\Psi_{{\bf k},{\bf r}}$, $\Psi_{{\bf k}',{\bf r}'}$, and $\Psi_{{\bf k}'',{\bf r}''}$ denote the
corresponding TT parametrizations. Since $r_i=1$, for every multi-index
$(j_0,\ldots,j_n)$ we have
$
\psi_{j_0,\ldots,j_n}
=
\bigl( A^{(j_0)}_0\cdots A^{(j_{i-1})}_{i-1} \bigr) \cdot 
\bigl( A^{(j_i)}_i\cdots A^{(j_n)}_n\bigr)
$, where both factors are $1\times 1$ matrices. Hence we may write
$
\psi_{j_0,\ldots,j_n}
=
\phi_{j_0,\ldots,j_{i-1}}\,
\xi_{j_i,\ldots,j_n},
$
where
$
\phi_{j_0,\ldots,j_{i-1}}
=
A^{(j_0)}_0\cdots A^{(j_{i-1})}_{i-1}
$
is a point in $V_{{\bf k}',{\bf r}'}$, and
$
\xi_{j_i,\ldots,j_n}
=
A^{(j_i)}_i\cdots A^{(j_n)}_n
$
is a point in $V_{{\bf k}'',{\bf r}''}$. Thus every tensor in the image of
$\Psi_{{\bf k},{\bf r}}$ lies in the Segre embedding of 
$
V_{{\bf k}',{\bf r}'}\times V_{{\bf k}'',{\bf r}''}.
$

Conversely, a point in the Segre product of $V_{{\bf k}',{\bf r}'}$ and 
$V_{{\bf k}'',{\bf r}''}$ is obtained by concatenating the corresponding TT cores along
the bond dimension $r_i=1$. Therefore the two parametrized images agree,
and taking Zariski closures proves the claim.
\end{proof}
From now on, we use the shorthand $(1)_n := (\underbrace{1,\dots,1}_{n})$ for sequences of repeated indices.
\begin{corollary} \label{cor:Segre-as-tensor-train}
For $\mathbf{k} = (k_0, k_1,\ldots, k_n)$ and $\mathbf{r} = (1)_n$ we have
$V_{\mathbf{k}, \mathbf{r}}=\mathbb{P}^{k_0-1} \times \cdots \times \mathbb{P}^{k_n-1}$, embedded by the Segre map. 
\end{corollary}
\begin{proof}
The matrices $A_i^{(j_i)}$  in the parametrization are $1 \times 1$ matrices: 
$A_i^{(j_i)} = (a_i^{j_i})$. Therefore,
$\psi_{j_0, j_1, \ldots, j_n} = a_0^{j_0} a_1^{j_1} \cdots a_n^{j_n}$. This is precisely the Segre embedding of 
$\mathbb{P}^{k_0-1} \times \cdots \times \mathbb{P}^{k_n-1}$.
\end{proof}
We next isolate a large, explicit family of tensor train varieties that are again Segre pro\-ducts. We say that indices $i$ with $r_i>1$ are \emph{separated} if no two such indices are adjacent; likewise, we call a \emph{separated block} a consecutive sequence $(i_1,\dots,i_t)$ with ${r_{i_1},\dots,r_{i_t}>1}$ such that no other $r_j>1$ is adjacent to the block.
\begin{theorem}\label{thm:TT-as-Segre}
 Let $\mathbf{k}=(k_0,\dots,k_n)$ and $\mathbf{r}=(r_1,\dots,r_n)$, with $r_0=r_{n+1}=1$. Suppose $\mathbf{r}$ has $\ell$ separated blocks of consecutive indices
$I_p=(a_p,\dots,b_p)$  such that $r_i>1$ iff $i\in \bigcup_p I_p$. 
Assume furthermore that for all $i \in \bigcup_p I_p$, either $r_i \geq r_{i-1}k_{i-1}$ or $r_{i-1} \geq r_i k_{i-1}$.
Then
 \[
V_{\mathbf{k},\mathbf{r}} \simeq
\mathbb{P}^{K_1-1}\times \cdots \times \PP^{K_\ell-1} \times \mathbb{P}^{k_{j_1}-1} \times \cdots \times \mathbb{P}^{k_{j_s} -1},
\]
where $K_p=\prod_{v=a_p-1}^{b_p} k_v$ for $p = 1,\dots,\ell$ and $j_{q}$ are such that $r_{j_q+1} = 1$ is not adjacent to two separated blocks for $q=1,\dots,s$. Moreover,
\[
\dim V_{\mathbf{k},\mathbf{r}} = \sum_{p=1}^\ell (K_p - 1) 
+ \sum_{q=1}^s (k_{j_q} - 1), \quad \deg V_{\mathbf{k},\mathbf{r}} = 
\frac{\left( \sum_{p=1}^\ell (K_p - 1) + \sum_{q=1}^s (k_{j_q} - 1) \right)!}
{\prod_{p=1}^l (K_p - 1)! \; \prod_{q=1}^s (k_{j_q} - 1)!}.
\]
\end{theorem}
\begin{proof}
For a block $I_p$, the corresponding TT cores have shape $(r_i \times k_i \times r_{i+1})$ for $i \in I_p$, and $r_{a_p-1}=r_{b_p+1}=1$. 
These ranks exceed~1 and either $r_i \geq r_{i-1}k_{i-1}$ or 
$r_{i-1} \geq r_i k_{i-1}$, so the block parameters 
span all coordinates in $\mathbb{P}^{K_p - 1}$, giving a 
rational reparametrization~onto~$\mathbb{P}^{K_p - 1}$.

By repeated application of Lemma~\ref{lem: bond_one_segre}, the global parametrization factors as the Segre product of the $\mathbb{P}^{K_p - 1}$ for $p \in \{1,\dots,\ell\}$ with the remaining $\mathbb{P}^1$ factors.  The dimension is the sum of the factor dimensions, and the degree of  $\mathbb{P}^{d_1}\times\cdots\times\mathbb{P}^{d_s}$ is
the multinomial coefficient $$\binom{d_1+\cdots+d_s}{d_1\,\cdots\,d_s},$$ yielding the claimed formulas.
\end{proof}
We are mostly interested in binary tensors ($k_i=2$) due to their prevalence in applications.
\begin{corollary} \label{cor: binaryTT-as-Segre}
Let $\mathbf{k}=(2)_{n+1}$ and $\mathbf{r}=(r_1,\dots,r_n)$.
Suppose there are
$\ell$ separated indices $i$ with $r_i>1$,
$t$ separated pairs $(i,i+1)$ with $r_i,r_{i+1}>1$,
and all remaining $r_j$ equal $1$.
Then
\[
V_{\mathbf{k},\mathbf{r}} \simeq
(\mathbb{P}^3)^\ell \times (\mathbb{P}^7)^t \times (\mathbb{P}^1)^{\,n-\ell-2t-s},
\]
where $s$ is the number of $r_j=1$ adjacent to two separated blocks.
In particular, the dimension of $V_{\mathbf{k}, \mathbf{r}}$ is $n+2\ell + 5t -s$ and its degree is the multinomial coefficient 
$$\frac{(n+2\ell+5t-s)!}{(3!)^\ell (7!)^t}.$$
\end{corollary}
\noindent
The dimensions of the varieties above also follow from the general formula in Corollary~\ref{cor:TT-dim}.

We now turn to equations. Lemma~\ref{lem: another definition of TT} shows that $V_{\mathbf{k},\mathbf{r}}$ is set-theoretically cut out by determinantal equations. The following conjecture, first suggested to us by Bernd Sturmfels in personal communication, asserts that these suffice scheme-theoretically as well.
\begin{conjecture}\label{conj: generating_ideal}
   For fixed $\mathbf{k}$ and $\mathbf{r}$, let $\psi$ be a $k_0 \times \cdots \times k_n$ tensor of indeterminates and let $\psi^i$ be the flattenings as above. Then the homogeneous prime ideal of $V_{\mathbf{k}, \mathbf{r}}$ in $\mathbb{C}[\psi]$ is
\[
\mathcal I(V_{\mathbf{k}, \mathbf{r}})
= \sum_{i=0}^{n-1} \Big\langle (r_{i+1} + 1)\textrm{-minors of }\psi^i \Big\rangle.
\]
Moreover, the union of $(r_{i+1} + 1)$-minors of flattenings $\psi^i$ for $i=0,\dots,n-1$ forms a Gröbner basis for the ideal $\mathcal I(V_{\mathbf{k}, \mathbf{r}})$.
\end{conjecture}
We collect some positive evidence in the case of binary tensors, in particular, in the case of Segre varieties whose 
Gr\"{o}bner bases  have been studied widely. We start with the following lemma; a proof is available in \cite[Chapter~3]{VBthesis}.
\begin{lemma}\label{lem: 2minorsGB}
Let $\mathbf{k}=(2)_{n+1}$ and $\mathbf{r}=(r_1,\dots,r_n)$ with $\ell$ separated indices where $r_i>1$, $t$ separated pairs $(i, i+1)$ where $r_i,r_{i+1}>1$, and all remaining $r_j=1$. Let
\[
\{i_1,\dots,i_{n-\ell-2t}\}=\{1,\dots,n\}\setminus \big(\{i:\ r_i>1\}\cup\{i,i+1:\ r_i,r_{i+1}>1\}\big).
\]
Then the union
\[
G_{i_1} \, \cup\ \cdots\ \cup \, G_{i_{n-\ell-2t}},
\]
where $G_{j}$ denotes the set of $2$-minors of $\psi^{j-1}$, is a minimal reduced Gröbner basis of $\mathcal I(V_{\mathbf{k},\mathbf{r}})$ with respect to any lexicographic or reverse lexicographic term order that respects the row–column structure of the flattenings $\psi^j$ for $j=0,\dots,n-1$.
\end{lemma}
\begin{remark}
In general the union above is not a universal lexicographic or universal reverse lexicographic Gröbner basis, see the example below.
\end{remark}
\begin{example}[Segre products of two projective spaces]
Let ${\bf k}=(k_0,k_1)$ and ${\bf r}=(1)$. Then $\mathcal I(V_{\mathbf{k},\mathbf{r}})$ is generated by the $2$-minors of a $k_0\times k_1$ matrix. By \cite[Theorem 3.9]{Restuccia2007OnCC}, these minors form a quadratic universal reverse-lex Gröbner basis, implying strong Koszulity of the coordinate ring. For $k_0,k_1>2$, they are not a universal Gröbner basis. For instance, for $(k_0,k_1)=(3,3)$ the Gröbner fan has 108 maximal cones, but only 96 correspond to reduced Gröbner bases consisting entirely of $2$-minors.
\end{example}
\begin{example}[Segre products of several projective spaces]
For $n>1$, the $2$-minors need not form a reverse-lex Gröbner basis under an arbitrary permutation of the variables. For ${\bf k}=(2,2,2)$ and ${\bf r}=(1)_2$, order the variables as
\[
\psi_{111} \succ \psi_{122} \succ \psi_{212} \succ \psi_{221} \succ \psi_{222} \succ \psi_{121} \succ \psi_{211} \succ \psi_{112}.
\]
The reduced reverse-lex Gröbner basis contains the cubic
$\psi_{222}\,\psi_{121}\,\psi_{211} - \psi_{221}^2\, \psi_{112}$, showing that $2$-minors alone do not form a Gröbner basis in this case.
\end{example}
We record a conjecture that represents a key step towards the proof of the Conjecture~\ref{conj: generating_ideal}.
\begin{conjecture}
Let $\mathbf{k}=(2)_{n+1}$ and $\mathbf{r}=(1,r_2,\dots,r_{n-1},1)$. Then the union
\[
G^{(r_1+1)}_1\ \cup\ \cdots\ \cup\ G^{(r_n+1)}_n
\]
of $(r_i+1)$-minors of the flattenings $\psi^{i-1}$ forms a Gröbner basis of $\mathcal I(V_{\mathbf{k},\mathbf{r}})$ with respect to the reverse-lexicographic order induced by a variable order compatible with the flattenings.
\end{conjecture}

\section{Birational Parametrization and TT manifolds} \label{sec:birational}
In this section, we present a one-to-one parametrization of a dense open set of the tensor train variety $V_{\mathbf{k}, \mathbf{r}}$ from a product of Grassmannians defined by
$$
\mathcal Z \coloneq 
\mathrm{Gr}(r_1,k_0) \times \mathrm{Gr}(r_2,r_1 k_1) \times \cdots \times \mathrm{Gr}(r_{n},r_{n-1} k_{n-1}) \times  \mathbb{P}(\C^{r_{n} \times k_n}).
$$
This open set is contained in  the \emph{tensor train manifold}, denoted by $V_{\mathbf{k},\mathbf{r}}^=$, which is the set of tensors in $V_{\mathbf{k},\mathbf{r}}$ whose $i$-th flattening has rank exactly $r_{i+1}$ for all $i = 0, \ldots, n-1$. 
The construction is inspired by the TT-cross algorithm~\cite{Oseledets2010}, itself an adaptation of the TT-SVD~\cite{Grasedyck2010, Oseledets2011} in which successive SVDs are replaced by so‑called skeleton decompositions of matrices~\cite{Tyrtyshnikov2000}.  
Skeleton (or $CUR$) decompositions factor an $m\times n$ matrix $A$ by selecting $r$ linearly independent columns and rows of $A$ as columns and rows of $C \in \C^{m \times r}$ and $R \in \C^{r \times n}$, respectively. In the literature, this decomposition is often generalized to the case where only rows or columns are chosen, resulting in an $XR$-decomposition or $CX$-decomposition, respectively~\cite{Drineas2006}.

In our algorithm, we use a special case of the $XR$-decomposition for a rank-$r$ matrix $A$, under the following assumption.
\begin{assumption}\label{ass:XR}
For a rank‑$r$ matrix \(A \in \mathbb{C}^{m \times n}\), the first \(r\) rows are linearly independent.
\end{assumption}
If Assumption~\ref{ass:XR} holds, the $XR$‑decomposition is defined as:
$$
(X,R) \coloneq \mathtt{XR}(A,r),
$$
where \(R \in \mathbb{C}^{r \times n}\) is formed from the first \(r\) rows of \(A\), and
$X = A R^\dagger (R R^\dagger)^{-1} \in \mathbb{C}^{m \times r}$. Invertibility of \(RR^\dagger\) is guaranteed by Assumption~\ref{ass:XR}. This decomposition satisfies $XR = AR^\dagger(RR^\dagger)^{-1}R = AP$, where $P$ is the orthogonal projection onto the rowspan of $A$. 
Since all rows of $A$ are in the rowspan of $A$, projecting onto the rowspan of $A$ leaves them unchanged, and hence $XR = AP = A$. 
Moreover, $X$ has the form
\begin{equation}\label{eq:Xgr}
X = \begin{bmatrix}
{\rm Id}_r \\
B
\end{bmatrix},
\end{equation}
with \(B \in \mathbb{C}^{(m-r) \times r}\).  
Such matrices parametrize a large open subset of \(\mathrm{Gr}(r,m)\) via their column spans, with the first block fixed to the identity to ensure a one‑to‑one parametrization.

Our algorithm gives a bijection from an open set $\mathcal U \subseteq V_{\mathbf k, \mathbf r}^=$ to an open set $\mathcal W \subseteq \mathcal Z$.
In order to write down our algorithm, we will first need to identify these open subsets.
The description of the open set $\mathcal W$ is straightforward.
Let $W_i$ denote the open set 
 \[
  W_i := \{\mathrm{span}(A) : A \in \mathbb{C}^{r_{i-1}k_{i-1} \times r_i},\ \text{first $r_i$ rows linearly independent}\} \subseteq \mathrm{Gr}(r_i,r_{i-1}k_{i-1})
  \]
for $i = 1,\ldots,n$ and let $\C_\ast^{r_{n} \times k_n} \subseteq \C^{r_{n} \times k_n}$ denote the open set of full rank matrices. 
Then the following set is open in $\mathcal Z$:
\begin{equation}
\mathcal W \coloneq W_1 \times \cdots \times W_{n} \times \C_\ast^{r_{n} \times k_n}.
\end{equation}
The set \(\mathcal{U}\) consists of tensors for which Assumption~\ref{ass:XR} holds in every step of the algorithm; its precise description appears in the proof of Theorem~\ref{thm:param}, where we show it is Zariski dense.

Similar to Section~\ref{sec: TTvarieties}, we denote the $i$-th flattening of a tensor  $\psi \in \bigotimes_{i=0}^n \mathbb{C}^{k_i}$ by 
$$
\psi^i = \mathtt{unfold}(\psi,[k_0 \cdots k_i,k_{i+1} \cdots k_n]) \in \C^{(k_0 \cdots k_i) \times (k_{i+1} \cdots k_n)}
$$
and the tensorization of a matrix $\psi^i  \in \C^{(k_0 \cdots k_i) \times (k_{i+1} \cdots k_n)}$ is denoted by
$$
\psi = \mathtt{tensorize}(\psi^i,[k_0,\ldots,k_n]) \in \bigotimes_{i=0}^n \mathbb{C}^{k_i}.
$$
With this, we are equipped to present the parametrization algorithm for a tensor with given TT-rank $\mathbf r = (r_1,\ldots,r_{n})$. The procedure is similar to the one in Lemma~\ref{lem: another definition of TT}; see Algorithm~\ref{alg:factorization}.

\smallskip

\begin{algorithm}[H]\label{alg:factorization}
\SetKwInOut{Input}{Input}
\SetKwInOut{Output}{Output}
\caption{Factorization $\mathtt{factorize}(\psi,\mathbf{r})$}
\Input{$\psi \in \mathcal{U} \subseteq V_{\mathbf{k},\mathbf{r}}^=$, TT‑rank $\mathbf{r}=(r_1,\dots,r_{n})$}
\Output{$(X_0,\dots,X_n) \in \mathcal{W} \subseteq \mathcal{Z}$}
$r_0=1$, $\phi = \psi$\;
\For{$i=0$ \KwTo $n-1$}{
    $\psi^i = \mathtt{unfold}(\phi, [r_{i}k_i,\, k_{i+1}\cdots k_n])$\;
    $(X_i,R_i) = \mathtt{XR}(\psi^i,r_{i+1})$\;
    $\phi = \mathtt{tensorize}(R_i,[r_{i+1},k_{i+1},\dots,k_n])$\;
}
$X_n = \phi/\phi_{11}$\;
\Return $(X_0,\dots,X_n)$
\end{algorithm}
\begin{remark}
Algorithm~\ref{alg:factorization} works only if the rank conditions in \eqref{eq: kr constraints} are satisfied. 
For any $i$ with $r_i > r_{i-1} k_{i-1}$, similarly to Remark~\ref{rem: min bond dims}, we replace $r_i$ by
$
\min\{r_i,\; r_{i-1}k_{i-1},\; r_{i+1}k_i\}
$ before running the algorithm. 
Note that the procedure can also be performed starting from the right-hand side, in which case one uses $CX$-decompositions instead of $XR$-decompositions.
\end{remark}
We now describe the inverse of Algorithm~\ref{alg:factorization}, namely the
\emph{decompression} procedure that reconstructs a TT tensor from its parametrized
form; see Algorithm~\ref{alg:compression}.

\begin{algorithm}[H]\label{alg:compression}
\SetKwInOut{Input}{Input}
\SetKwInOut{Output}{Output}
\caption{Decompression $\mathtt{decompress}(X_0,\dots,X_n)$}
\Input{$(X_0,\dots,X_n) \in \mathcal{W} \subseteq \mathcal{Z}$}
\Output{$\psi \in \mathcal{U} \subseteq V_{\mathbf{k},\mathbf{r}}^=$}
${\bf r} := (1,r_1,\dots,r_{n},1)$, $\psi := X_n$\;
\For{$i = n-1$ \KwTo $0$}{
    $R_i = \mathtt{unfold}(\psi,[r_{i+1},\, k_{i+1} \cdots k_n])$\;
    $\psi^i = X_i R_i$\;
    $\psi = \mathtt{tensorize}(\psi^i,[r_{i},\, k_i,\dots,k_n])$\;
}
\Return $\psi$
\end{algorithm}
\begin{theorem}
\label{thm:param}
Algorithm~\ref{alg:compression} defines a one-to-one parametrization from $\mathcal Z$ to $V_{\mathbf{k}, \mathbf{r}}^{=}$. 
The restriction $\mathcal Z_{\mathbb R} \to V_{\mathbf{k}, \mathbf{r}, \mathbb R}$ to the real points is birational. 
\end{theorem}
\begin{proof}
 The subset $\mathcal U \subseteq V_{\mathbf k,\mathbf r}^=$ consists of all tensors satisfying Assumption~\ref{ass:XR} for each $XR$-factorization in Algorithm~\ref{alg:factorization}, i.e.,  the first $r_{i+1}$ rows of $\psi^i$ (Line 3) are linearly independent for $i = 0, \ldots, n-1$.
 In other words, the matrix $R_i$ (Line 4) must have full rank, that is, it must have some nonvanishing maximal minor, for all $i = 0, \ldots, n-1$. 
 We argue that this condition is algebraic in the entries of the input tensor $\psi$. 
 
 Indeed, the matrix $R_0$ is defined as the first $r_1$ rows of $\psi^0$, which is just a reshaping of the input tensor $\psi$. 
 Hence all entries of $R_0$ are entries of the input tensor $\psi$. 
 Similarly, since $\psi^1$ and $R_1$ are both flattenings of the same tensor, all entries of $R_1$ are also entries of $R_0$, and thus of the input tensor $\psi$. 
 Proceeding in this manner, we observe that the entries of $R_0, \ldots, R_{n-1}$ are always entries in the original $\psi$.
 Thus the condition that $R_i$ has a nonvanishing maximal minor for $i = 0, \ldots, n-1$ is an algebraic condition in the entries of the input tensor $\psi$. 
 Hence, the open set $\mathcal U$ is Zariski dense in $V_{\mathbf k,\mathbf r}^=$.
 
We now prove that Algorithm~\ref{alg:factorization} is a bijection from the open set $\mathcal U \subseteq V_{\mathbf k,\mathbf r}^=$ to the open set $\mathcal W \subseteq \mathcal Z$ by showing that Algorithm~\ref{alg:compression} is its inverse.
We proceed by induction on $n$. 
For $n = 1$, the claim follows from the $XR$-decomposition. 
 For general $n$, let $(X_0,\ldots,X_n) \in \mathcal W$ and let $\psi = \mathtt{decompress}(X_0,\ldots,X_n)$.
 After the first iteration of the \textbf{for}-loop in Algorithm~\ref{alg:factorization}, we have recovered $X_0$ and $\psi$ is now $\mathtt{decompress}(X_1,\ldots,X_n)$; this is because this first iteration of the \textbf{for}-loop in Algorithm~\ref{alg:factorization}
 undoes the last iteration of the \textbf{for}-loop in Algorithm~\ref{alg:compression}. 
 By induction, the output of $\mathtt{factorize}(\mathtt{decompress}(X_1,\ldots,X_n))$ matches $(X_1, \ldots, X_n)$, and we have reached our conclusion. 
 The argument for the reverse direction is analogous.
\end{proof}

\begin{remark}
We believe that the statement that $\mathcal Z$ parameterizes $V_{{\bf k}, {\bf r}}^=$ is contained in~\cite{Falco2023}, somewhat enigmatically. 
However, to the best of our knowledge, the resulting one-to-one parametrization of tensor train manifolds in Algorithm~\ref{alg:compression} is derived here for the first time.
\end{remark}
The dimension of the manifold $V_{\mathbf k,\mathbf r}^=$ is known in the tensor decomposition literature; see, for example, \cite{Holtz2012b}. It can also be derived from \cite{BBM}, where more general tensor network varieties were studied.
We present it here as a corollary of Theorem~\ref{thm:param}.

\begin{corollary} \label{cor:TT-dim}
The tensor train variety $V_{\mathbf k,\mathbf r}$ has dimension  
$$\mathrm{dim}(V_{\mathbf k,\mathbf r}) = \left( \sum_{i=0}^{n-1} (r_{i}k_i - r_{i+1}) \cdot r_{i+1} \right) + r_nk_n-1.$$
\end{corollary}

\begin{example}
We parametrize the tensor train variety $V_{\mathbf{k},\mathbf{r}} \cong \mathbb P^5 \times \mathbb P^1$. 
We map from the open set $\mathcal W \subseteq \mathrm{Gr}(2,3) \times \mathrm{Gr}(1,4) \times \mathbb P^1$. 
The image of a point
\begin{align*}
    X_0 = 
    \begin{footnotesize}
    \begin{pmatrix}
        1 & 0 & a \\ 0 & 1 & b
    \end{pmatrix}^T,
    \end{footnotesize}
    &&   
    X_1 = 
     \begin{footnotesize}
    \begin{pmatrix}
        1 & c & d & e
    \end{pmatrix}^T,
    \end{footnotesize}
    &&
    X_2 = 
    \begin{footnotesize}
    \begin{pmatrix}
        1 & f
    \end{pmatrix}
        \end{footnotesize}
\end{align*}
in $\mathcal W$ under Algorithm~\ref{alg:compression} is $\psi$ with slices
\begin{align*}
    \psi_{(1,:,:)}  = 
    \begin{footnotesize}
    \begin{pmatrix}
        1 & f \\ c & cf
    \end{pmatrix}
    \end{footnotesize}
    &&
    \psi_{(2,:,:)}  = 
    \begin{footnotesize}
    \begin{pmatrix}
        d & df \\ e & ef
    \end{pmatrix}
    \end{footnotesize}
    &&
    \psi_{(3,:,:)}  = 
    \begin{footnotesize}
    \begin{pmatrix}
        a+bd & (a+bd)f \\ ac+be & (ac+bd)f
    \end{pmatrix}.
    \end{footnotesize}
\end{align*}
We illustrate Algorithm~\ref{alg:factorization} with an explicit numerical example.
Consider the $3 \times 2 \times 2$ tensor $\psi$ of TT-rank $(2,1)$ whose slices are 
\begin{align*}
    \psi_{(1,:,:)} = 
    \begin{footnotesize}
    \begin{pmatrix} 2 & 3 \\ 4 & 6 
    \end{pmatrix}\end{footnotesize}, 
    \quad
    \psi_{(2,:,:)} = 
    \begin{footnotesize}
    \begin{pmatrix} 6 & 9 \\ 8 & 12 \end{pmatrix}\end{footnotesize}, 
    \quad
    \psi_{(3,:,:)} =
    \begin{footnotesize}
    \begin{pmatrix} -2 & -3 \\ -4 & -6 \end{pmatrix}\end{footnotesize}.
\end{align*}
Its first unfolding and decomposition in Algorithm~\ref{alg:factorization} yield
\begin{align*}
    \psi^0 = \begin{footnotesize}\begin{pmatrix}
        2 & 3 & 4 & 6\\
        6 & 9 & 8 & 12\\
        -2 & -3 & -4 & -6
    \end{pmatrix}\end{footnotesize},
    \quad
     X_0 = \begin{footnotesize}\begin{pmatrix}
     1 & 0 \\
     0 & 1\\
     -1 & 0
     \end{pmatrix}\end{footnotesize},
    \quad
    R_0 = \begin{footnotesize}\begin{pmatrix}
    2 & 3 & 4 & 6\\
    6 & 9 & 8 & 12
     \end{pmatrix}\end{footnotesize}.
\end{align*}
We then tensorize $R_0$ to obtain the $2 \times 2\times 2$ tensor $\phi$ with slices
\begin{align*}
    \phi_{(:,:,1)} = 
    \begin{footnotesize}\begin{pmatrix}
       2 & 4 \\ 6 & 8
    \end{pmatrix}\end{footnotesize}, \quad
    \phi_{(:,:,2)} = 
    \begin{footnotesize}\begin{pmatrix}
        3 & 6\\ 9 & 12
    \end{pmatrix}\end{footnotesize}.
\end{align*}
This tensor unfolds and factors as
\begin{align*}
    \psi^1 = \begin{footnotesize}\begin{pmatrix}
    2 & 4 & 6 & 8\\
    3 & 6 & 9 & 12
    \end{pmatrix}^T\end{footnotesize},\quad
    X_1 = \begin{footnotesize}\begin{pmatrix}
        1 & 2 & 3 & 4
    \end{pmatrix}^T\end{footnotesize},\quad
    R_1 = \begin{footnotesize}\begin{pmatrix}
    2 & 3
    \end{pmatrix}\end{footnotesize}.
\end{align*}
Thus the output of the algorithm is 
\begin{align*}
X_0 = 
\begin{footnotesize}\begin{pmatrix}
   1 & 0 & 0 \\ 0 & 1 & -1
\end{pmatrix}^T\end{footnotesize},  \quad
X_1 = 
\begin{footnotesize}\begin{pmatrix}
    1& 2 &3 &4
\end{pmatrix}^T\end{footnotesize}, \quad
X_2 = 
\begin{footnotesize}\begin{pmatrix}
    1 &3/2
\end{pmatrix}\end{footnotesize}.
\end{align*}
\end{example}

\section{Numerical Experiments}\label{sec:numerics}
In our last section, we return to TT varieties and perform numerical computations for some small instances. The primary numerical task is the computation of all complex critical points of \eqref{eqn:optproblem} over a TT variety. For this, we use homotopy continuation and monodromy methods. We report the results of our computations for small determinantal varieties as well as for the first few instances of TT varieties that are Segre products by Theorem \ref{thm:TT-as-Segre}; see Tables~\ref{tab: RRdeg determnantal} and \ref{tab:RRdegtensors}. 
In many of these cases, the RR degrees themselves can be derived from the formulae obtained in \cite{SSW25}; see, in particular, Corollaries 5.2 and 5.3. 
However, we compute not only the degrees, but the actual critical points. 
One of our observations is that the number of local extrema is typically much smaller than the corresponding RR degree. 

We also report the degree of the nonisotropic part of the RR discriminant (see Section~\ref{sec:RR-dis}) for some determinantal and TT varieties; see Tables~\ref{tab:RRdegtensors} and \ref{tab:RR-disc-degs}. The take-away message there is that these degrees are high, implying that the computation of the actual discriminant polynomials are currently out of reach. 

Finally, we try to gauge the quality of the low energy state solutions produced by the DMRG and ALS algorithms by comparing to the complete list of local and global minima produced by our numerical methods. 
We conclude that ALS frequently gets stuck in local minima with suboptimal energy, although this happens less often for physical Hamiltonians; see Table~\ref{tab: ALS M}. The DMRG method does not converge to critical points of~\eqref{eqn:optproblem} at all.

\subsection{Hamiltonians}

In our experiments, whenever we look at a specific Hamiltonian $H$, we consider two types: 

\begin{enumerate}
    \item \textbf{Random symmetric real matrices.} 
    
    We generate a symmetric real matrix $H \in \mathbb{R}^{(k_0 \cdots k_n) \times (k_0 \cdots k_n)}$ by drawing entries from the standard normal distribution and symmetrizing. 
    For the ALS and DMRG experiments, $H$ is then expressed in tensor operator form, see \cite{Oseledets2011}.
    \item \textbf{Random Hamiltonians in second quantization.}  
    
    In second-quantized form, the Schrödinger operator has the structure (see \cite{Szalay2015}):
    \begin{equation}
        H = \sum_{i,j} t_{ij} \mathbf a_i^\dagger \mathbf a_j + \sum_{i,j,k,l} v_{ijkl} \mathbf a_i^\dagger \mathbf a_j^\dagger \mathbf a_k \mathbf a_l,
    \end{equation}
    where $\mathbf a_i^\dagger = \mathbf s \otimes \cdots \otimes \mathbf s \otimes \mathbf a^\dagger \otimes \mathbf I \otimes \cdots \otimes \mathbf I$ and $\mathbf a_i = \mathbf s \otimes \cdots \otimes \mathbf s \otimes \mathbf a \otimes \mathbf I \otimes \cdots \otimes \mathbf I$ are the {\em creation} and {\em annihilation operators}, respectively. Here, we have
    \begin{equation}
        \mathbf s =  \begin{footnotesize}\begin{pmatrix}
            1 & 0 \\
            0 & -1
        \end{pmatrix} \end{footnotesize}, \qquad
        \mathbf a =  \begin{footnotesize}\begin{pmatrix}
            0 & 1 \\
            0 & 0
        \end{pmatrix} \end{footnotesize}, \qquad 
        \mathbf I =  \begin{footnotesize}\begin{pmatrix}
            1 & 0 \\
            0 & 1
        \end{pmatrix} \end{footnotesize}.
    \end{equation}
    The coefficients $t_{ij}, v_{ijkl} \in \mathbb{R}$ for $i,j,k,l \in \{0,1,\dots,n\}$ are chosen randomly from the standard normal distribution. 
    In this case, $\mathbf{k} = (2)_{n+1}$, and $H$ is  symmetric.
\end{enumerate}

\subsection{Homotopy Continuation Computations}
A central task in our setting is to solve polynomial systems that encode the optimization problem \eqref{eqn:opt-param} on tensor trains. 
There are many ways to write down such a system, for example by using Proposition~\ref{prop:Icrit-RR}, Proposition~\ref{prop:RR degree-rational}, or \eqref{eqn:critical-ideal}.

Our main computational tool is \emph{numerical homotopy continuation}, see  \cite{Sommese:Wampler:2005}. This is a numerical method for finding all isolated solutions to a system of polynomial equations. Concretely, let $\bfx=(x_1,\dots,x_n)$ be variables and $\bfp=(p_1,\dots,p_k)$ parameters, and consider a family of polynomial systems
\[
\mathcal F(\bfx;\bfp)=\big(f_1(\bfx;\bfp),\ldots,f_n(\bfx;\bfp)\big)\in\mathbb{C}[\bfx]^n,\quad \text{ with } \; \bfp\in\mathbb{C}^k.
\]
Choose two parameter values $\bfq_1,\bfq_2\in\mathbb{C}^k$ such that
the target system is $F(\bfx)=\mathcal F(\bfx;\bfq_1)$, and
the start system $G(\bfx)=\mathcal F(\bfx;\bfq_2)$ has easily computable solutions.
The \emph{parameter homotopy}
\[
H(\bfx,t):=\mathcal  F\big(\bfx;(1-t)\bfq_1+t\bfq_2\big),\quad t\in[0,1],
\]
deforms $G(\bfx)$ at $t=1$ to $F(\bfx)$ at $t=0$. The parameter continuation theorem \cite{BorovikBreiding, MS1989} guarantees that the number of regular isolated solutions in $\mathbb{C}^n$ is constant along $t\in(0,1]$. Starting from solutions of $G(\bfx)$, we track solution paths to $t=0$ using predictor–corrector schemes that numerically integrate the Davidenko ODE
\[
\big(\partial_{\bfx} H(\bfx,t)\big)\,\dot{\bfx} + \partial_t H(\bfx,t)=0.
\]
If the start system $G(\mathbf{x})$ has at least as many regular isolated solutions as the target $F(\mathbf{x})$ and all solution paths arrive at $t=0$, then we obtain all solutions of $F(\mathbf{x})$. 
When $G(\mathbf{x})$ has exactly as many solutions as $F(\mathbf{x})$, the homotopy is called \emph{optimal}.

Another numerical method for solving polynomial systems is \emph{monodromy}, introduced in~\cite{monodromy1}. 
Starting from a general parameter value $\mathbf{p}_0$ and a set of known solutions to ${\mathcal F(\mathbf{x};\mathbf{p}_0)=0}$, 
one traces each solution along closed loops in the parameter space based at $\mathbf{p}_0$ using parameter homotopy. 
These loops induce permutations of the solution set, and running random loops often uncovers previously unknown solutions. 
When the total number of solutions to a general system in the family is known, monodromy has a stopping criterion. Otherwise, suitable heuristics may be applied. Both methods are implemented in the \texttt{Julia} \cite{Julia} package  \texttt{HomotopyContinuation.jl} \cite{hc}.

We combine the two homotopy continuation methods introduced above to solve the optimization problem~\eqref{eqn:opt-param}, 
in particular to compute RR degrees and to numerically determine the degrees of RR discriminants. 
All algorithms are implemented in our \texttt{Julia} package \texttt{TensorTrainOptimization.jl} using \texttt{HomotopyContinuation.jl} and \texttt{Oscar.jl} \cite{OSCAR}, 
which is provided together with the complete dataset of our experiments at:
\begin{center}
    \url{https://zenodo.org/records/17777072}.
\end{center}
\begin{remark}
The experiments were run on a machine with two 12-core Intel Xeon E5-2680~v3 processors at 2.5 GHz and 512 GB RAM.     
\end{remark}
\begin{table}[h]
\centering
 \footnotesize
\begin{tabular}{|c|c|*{7}{c}|}
\toprule
$r$ & $m \backslash n$ & $2$ & 3 & 4 & 5 & 6 & 7 & 8  \\
\midrule
1 & 2 & 8 & 18 & 32 & 50 & 72 & 98 & 128 \\
1 & 3 & & 61 & 148 & 295 & 518 & 833 & 1256\\
1 & 4 & & & 480 & 1220 & 2624 & 5012 & 8768 \\
1 & 5 & & & &  3881 & 10166 & 23051 & 46856  \\
\midrule
2 & 2 & 4 & 6 & 8 & 10 & 12 & 14 & 16  \\
2 & 3 & & 154 & 448 & 970 & 1784 & 2954 & 4544   \\
2 & 4 & & & 5840 & 24924 & 74775 & 182306 &  386716\\
\midrule
3 & 3 &  & 9 & 12 & 15 & 18 & 21 & 24 \\
3 & 4 & &  & 2368 & 7340 & 17552 & 35755 & 65280  \\
3 & 5 & & & & 460351  & & & \\
\bottomrule
\end{tabular}
\caption{RR degrees of determinantal varieties of $m \times n$ matrices with rank $\leq r$.}
\label{tab: RRdeg determnantal}
\end{table}
\begin{example}[RR degrees]\label{ex: RRdegrees}
We work with the Rayleigh–Ritz optimization problem~\eqref{eqn:optsphere} formulated on the sphere. 
Specifically, we apply Algorithm~\ref{alg:factorization} to birationally parametrize a tensor $\psi$ 
in tensor train format for given $\mathbf{k}$ and $\mathbf{r}$. 
The function \texttt{get\_T} returns this tensor as a vector of polynomials in 
$\dim V_{\mathbf{k}, \mathbf{r}}$ parameters. 
The critical equations of \eqref{eqn:optsphere} form a polynomial system, which we solve via the monodromy method using the entries of $H$ as parameters. 
The monodromy computation produces twice the RR degree many solutions; see Section~\ref{sec:BW}. 

We perform these computations for low-rank matrices up to size~$5$ and for binary tensor train varieties up to order~$6$, 
with results reported in Tables~\ref{tab: RRdeg determnantal} and~\ref{tab:RRdegtensors}.
For instance, the following command computes the first entry of Table~\ref{tab: RRdeg determnantal}:
\begin{verbatim}
get_RRdegree_monodromy([2,2], [1])
\end{verbatim}
The largest number 460351 in Table \ref{tab: RRdeg determnantal} was computed in approximately 205 hours, 
using 29.465~GB of memory. 
\begin{proposition}
    The smallest tensor train variety which is not a Segre product is $V_{{\bf k}, {\bf r}}$ where ${\bf k} = (2)_6$ and ${\bf r} = (1, (2)_3, 1)$.
    Its RR degree is $\geq 691,127$.
\end{proposition}
We computed this lower bound by running the monodromy method as described above. 
The computation did not terminate after 60 days. 
\end{example}
\begin{table}[h]
\centering
 \footnotesize
\begin{tabular}{|c|c|c|c|c|c|r|}
\toprule
$\mathbf{k}$ & $\mathbf{r}$ & $\mathrm{codim}$ & RR deg & time & Disc.~deg & time\\
\midrule
$(2)_3$ & $(1)_2$    &   5      & 48   & $14s$ & 360& $1.8m$ \\
$(2)_4$ & $(1)_3$    &   12    & 384   & $90s$ & 5438 & $3h$\\
$(2)_4$ & $(1,2,1)$   &   10     & 352   & $59s$ & 6000 & $6.4h$\\
$(2)_5$ & $(1)_4$   &   27 & 3840  & $26m$ &  & \\
$(2)_5$ & $(1,2,1,1)$    &  25   & 4608  & $33m$ && \\
$(2)_5$ & $(1,2,2,1)$ & 22 & 2752  & $1h$ && \\
$(2)_6$ & $(1)_5$  &  58  & 46080 & $24h$ && \\
$(2)_6$ & $(1,2,(1)_3)$  &  56   & 69120  & $49h$ &&\\
$(2)_6$ & $(1,1,2,1,1)$   & 56   & 69120  & $37h$ &&\\
$(2)_6$ & $(1,2,2,1,1)$  &  53   & 67583  & $153h$ &&\\
$(2)_6$ & $(1,2,1,2,1)$  &  55   & 134656  & $78h$ &&\\
\bottomrule
\end{tabular}
\caption{RR degrees and RR discriminant degrees for small tensor train varieties.
}
\label{tab:RRdegtensors}
\end{table}
\begin{example}[Optimization on Tensor Trains] \label{ex: optimizationTT}
The computation proceeds as in Example~\ref{ex: RRdegrees}: 
we first compute twice the RR degree many critical points for a random symmetric Hamiltonian via monodromy, 
and then use parameter homotopy to track them to the critical points of a given Hamiltonian~$H$. 
For random real symmetric $H$, this parameter homotopy is still optimal. 

The function \texttt{TT\_optimization\_monodromy} returns the RR degree, 
the number of real critical points, and vectors of all complex and real solutions. 
Applying \texttt{get\_global\_extrema} to the real solutions yields the global minimum and the corresponding ground state energy. 
The function \texttt{return\_local\_min\_max} identifies all real critical points that are local minima or maxima 
by checking whether the Hessian of the Rayleigh–Ritz quotient is positive or negative definite.
We find that the number of local extrema is typically much smaller than the RR degree.
For example, the following command
\begin{verbatim}
TT_optimization_monodromy([2,2,2,2], [1,2,1]; H = M1, RRdeg = 352)   
\end{verbatim}
computes $352$ critical points for the tensor train variety 
$\mathbb{P}^3 \times (\mathbb{P}^1)^2$ and the given Hamiltonian~$M_1$.
The global minimum value is $-3.5692$, with only $8$ local extrema among the $352$ critical points.
These extremal values, computed using \texttt{return\_local\_min\_max}, are listed in Table~\ref{tab:extremal-values}.
We repeated this for four different real symmetric matrices $M_2$-$M_5$ and observe the same pattern of very few local extrema. When we compute the local extrema of five random Hamiltonians $H_1$-$H_5$ in second quantization, even fewer extrema can be observed. In fact, for $H_1, H_2$ and $H_4$ there is only one global minimum and maximum. For the last matrix $H_5$, we did not obtain any extrema and expect that the solution lies in the complement of the open set $\mathcal W$ defined in Algorithm~\ref{alg:factorization}, so our homotopy algorithm cannot reach it.
All ten matrices in this example are available at the \texttt{Zenodo} website \cite{zenodo_our}.
\end{example}
\begin{table}[h]
\centering
 \footnotesize
\begin{tabular}{|c|l|}
\toprule
$M$ & \textbf{Extremal values} \\ \midrule

$M_1$ & min: $-3.5692$, $-3.2364$, $-2.4927$, $-1.9931$; \\
   & max: $2.4383$, $3.4325$, $3.4346$, $3.4475$ \\ \hline
$M_2$ & min: $-3.0762$, $-3.023$; \\
   & max: $1.1842$, $2.352$, $2.4979$, $2.5777$ \\ \hline
$M_3$ & min: $-3.8648$, $-2.8834$; \\
   & max: $1.7402$, $1.9833$, $2.0849$, $3.2976$ \\ \hline
$M_4$ & min: $-4.2191$, $-2.4659$, $-2.2213$; \\
   & max: $1.9474$, $2.6817$, $3.1394$ \\ \hline
$M_5$ & min: $-4.1527$, $-3.2442$; \\
   & max: $0.7862$, $1.9889$, $2.1986$, $3.7494$, $3.8863$ \\
\bottomrule
\end{tabular}\qquad
\begin{tabular}{|c|l|}
\toprule
$H$ & \textbf{Extremal values} \\ \midrule
$H_1$ & min: $-11.0834$; \\
   & max: $1.2403$ \\ \hline
$H_2$ & min: $-5.7333$; \\
   & max: $8.0021$ \\ \hline
$H_3$ & min: $-2.8502$, $-2.3294$; \\
   & max: $5.4079$, $6.4597$ \\ \hline
$H_4$ & min: $-2.1645$; \\
   & max: $6.0279$ \\ \hline
$H_5$ & min: - \\
   & max: - \\
\bottomrule
\end{tabular}
\caption{Minimum and maximum values on $\PP^3 \times (\PP^1)^2$ for Hamiltonians $M_1$-$M_5$ and $H_1$-$H_5$.}
\label{tab:extremal-values}
\end{table}

Application-driven Hamiltonians are typically very sparse, with many zero entries. Specializing parameters within the family of systems can only decrease the number of complex solutions, so the number of complex critical points for physical Hamiltonians is typically smaller than the RR degree.
\begin{example}[Physical Hamiltonians]
Let ${\bf k} = (2)_3$ and ${\bf r} = (1)_2$. Consider the Hamiltonian
\[
H = 
\begin{footnotesize}
\begin{pmatrix}
0 &     0 &     0 &     0 &     0 &     0 &     0 &     0\\
    0 &    -0.5 & -0.1 &   0 &     0 &     0 &     0 &     0\\
    0 &    -0.1 &  -0.5 &  -0.1 &   0 &     0 &     0 &     0\\
    0 &     0 &    -0.1 &  -0.5 &   0 &     0 &     0 &     0\\
    0 &     0 &     0 &     0 &    -1 &  -0.1 &   0 &     0\\
    0 &     0 &     0 &     0 &    -0.1 &  -1 &  -0.1 &   0\\
    0 &     0 &     0 &     0 &     0 &    -0.1 &  -1 &  0\\
    0 &     0 &     0 &     0 &     0 &     0 &     0 &    -1.5 
\end{pmatrix}.
\end{footnotesize}
\]
The RR degree of the corresponding $V_{\mathbf{k}, \mathbf{r}}$ is 48, and for this particular Hamiltonian \eqref{eqn:optproblem} has 48 critical points. 
However, only 44 of them are in the open set  $\mathcal{W}$ defined in Algorithm \ref{alg:factorization}. 
In particular, the global minimum is one of the four points not in $\mathcal W$.
More computations with physical Hamiltonians are available at the \texttt{Zenodo} website \cite{zenodo_our}.
\end{example}
\begin{example}[RR discriminant degrees]
 We compute the degrees of the nonisotropic component of the RR discriminant from Section~\ref{sec:RR-dis}.
To this end, we model the parametric ramification locus using the Lagrange multiplier equations from Example~\ref{ex: RRdegrees}, augmented by the condition that their Jacobian is singular.
The resulting polynomial system has a positive-dimensional solution set, so we intersect it with a general affine line in the $H$-space parametrized by~$t$.
The projection of the solutions onto the $t$-coordinate yields a linear section of the main component of the RR discriminant, and the number of such points equals its degree.
This allows us to compute the degree of the discriminant without computing the discriminant itself, thereby avoiding a costly symbolic computation.
We report these degrees for higher-order tensors in Table~\ref{tab:RRdegtensors} and for determinantal varieties in Table~\ref{tab:RR-disc-degs}.
\end{example}
\begin{table}[h]
    \centering
     \footnotesize
    \begin{tabular}{|c|c|*{6}{c}|}
\toprule
$r$ & $m \backslash n$ & $2$ & 3 & 4 & 5 & 6 & 7   \\
\midrule
1 & 2 & 24 & 96 & 240 & 480 & 840 & 1344 \\
1 & 3 & & 648 & 2304 & 6000 & 12960 & 24696 \\
\midrule
2 & 3 & & 2430 & 10944 &  &  &     \\
\bottomrule
\end{tabular}
    \caption{Degree of the nonisotropic part of the RR discriminant $m \times n$ matrices of rank $\leq r$.}
    \label{tab:RR-disc-degs}
\end{table}

\subsection{DMRG and ALS}

In practice, the ground state energy of the Hamiltonian $H$ is computed by solving \eqref{eqn:optproblem} using the DMRG method \cite{Szalay2015}. For its derivation, we note that Algorithm~\ref{alg:compression} can be naturally generalized to take any tensor train $(A_0,\ldots,A_n) \in \prod_{i=0}^n \C^{r_i k_i \times r_{i+1}}$ as an input, where $A_i$ is the collection of the $(r_i\times r_{i+1})$ parameter matrices $A_i^{(j_i)}$. We denote this generalized map by
\begin{equation}
    \tau : \prod_{i=0}^n \C^{r_i k_i \times r_{i+1}} \rightarrow \C^{k_0 \times \cdots \times k_n}.
\end{equation}
The DMRG is a two-site version of the ALS \cite{Holtz2012a}, which we  introduce first. A random TT tensor is chosen as an initial guess. In each outer iteration loop (often called a {\em sweep}), every component $A_i$ is updated independently and successively, starting at $A_0$ and ending at $A_n$. Typically, the sweep is then performed backwards (from $A_n$ to $A_0$). These forward and backward sweeps are repeated until convergence.

In the $i$-th step of each sweep, only the component $A_i$ is updated. The algorithm ensures that the left components $A_0,\ldots,A_{i-1}$ are left-unitary, that is, $A_j^\dagger A_j = I_{r_{j+1}}$ for $j = 0,\ldots,i-1$, and the right components $A_{i+1},\ldots,A_n$ are right-unitary: $B_jB_j^\dagger = I_{r_j}$, where $B_j = \mathtt{unfold}(A_j,[r_j,k_jr_{j+1}])$ for $j = i+1,\ldots,n$. The solution of the subproblem 
\begin{equation}\label{eqn:ALSsubproblem}
\tilde A_i = 
\operatorname*{argmin}_{\substack{Z \in \C^{r_i k_i \times r_{i+1}}}}
\frac{
  \langle \tau(A_0,\ldots,Z,\ldots,A_n),\,
          H\,\tau(A_0,\ldots,Z,\ldots,A_n) \rangle
}{
  \langle \tau(A_0,\ldots,Z,\ldots,A_n),\,
          \tau(A_0,\ldots,Z,\ldots,A_n) \rangle
}.
\end{equation}
can be found by computing the eigenvectors of a significantly reduced Hamiltonian using conventional methods. In a forward sweep, $\tilde A_i$ is subsequently left-orthogonalized (for example, via a QR decomposition, setting $A_i \leftarrow Q$) to prepare for the update of $A_{i+1}$.

\begin{example}[Correctness of ALS]
 For the ten matrices from Example \ref{ex: optimizationTT}, we computed all critical points and classified the local minima and maxima (Table \ref{tab:extremal-values}). We then ran the ALS algorithm 1000 times per matrix, and recorded the approximations of the ground energy attained at convergence. For ${\bf k}=(2)_4$ and ${\bf r}=(1,2,1)$, these outcomes are summarized in Table \ref{tab: ALS M}, where in the second column we also list the smallest and largest eigenvalues on the full space. Here, $V_i$ means an approximation to the lowest eigenvalue returned by the ALS method and $C_i$ is the number of runs out of 1000 that produced $V_i$.

\begin{table}[h]\!\!\!
    \centering
    \footnotesize
\begin{tabular}{|c|c|c|c|c|c|c|c|c|c|c|}
\toprule
$M$ & Extremal eigenvalues  & $V_1$ & $C_1$ & $V_2$ & $C_2$ & $V_3$ & $C_3$ & $V_4$ & $C_4$ \\ \midrule
1 & -4.7543, 2.9994 & -3.5692 & 336 & -3.2364 & 332 & -2.4927 & 204  & -1.9931 & 128 \\
2 & -4.6467, 2.5529 & -3.0231 & 568  & -3.0762 & 432 &  &  &  &  \\
3 & -5.4243, 3.3603 & -3.8648 & 616 & -2.8834 & 384 &  &  &  &  \\
4 & -5.6478, 3.0117 & -4.2191 & 710 & -2.2213 & 146 & -2.4659 & 144  &  &  \\
5 & -4.7984, 3.9884 & -4.1527 & 628 & -3.2443 & 372 &  &  &  &  \\
\bottomrule
\end{tabular}
    \caption{ALS statistics for $\mathbf{k} = (2,2,2,2)$ and $\mathbf{r} = (1,2,1)$ for symmetric matrices $M_1$-$M_5$.}
    \label{tab: ALS M}
\end{table}

Comparing Tables \ref{tab:extremal-values} and \ref{tab: ALS M} shows that ALS can converge to any local minimum with nonzero probability. Notably, the global minimum is not necessarily the most frequently observed outcome (see the second row of Table \ref{tab: ALS M}) and its energy can be very far from the ground state energy in the full space.
Moreover, as the problem size increases, the number of local minima increases, and hence, empirically, ALS is increasingly likely to become trapped in a spurious local minimum. 
However, for physical Hamiltonians in second quantization, the spectrum typically contains fewer local minima, and ALS correspondingly performs better. For the matrices $H_1, H_2$ and $H_4$, there is only one global minimum, resulting in a 100\% success rate of ALS. Experiments for larger cases are available online at \cite{zenodo_our}.
\end{example}

Since the initial guess determines the size of each component, the ALS algorithm is automatically constrained to the TT manifold $V_{\mathbf k, \mathbf r}^=$ and the TT rank remains fixed. In contrast, the DMRG method is rank adaptive. The initial guess is typically chosen to have TT rank~$\bf 1$, i.e., $r_0 = \cdots = r_n = 1$. Here, all but two neighboring components (say $A_i$ and $A_{i+1}$) are fixed and orthogonalized as above. The subproblem~\eqref{eqn:ALSsubproblem} is then solved for these two components simultaneously and {\em without imposing the any constraint on $r_{i+1}$}. This yields a merged double component $\tilde A_{i,i+1} \in \C^{r_ik_i \times k_{i+1}r_{i+2}}$ that is subsequently decomposed to produce new $A_i$ (left-orthogonal) and $A_{i+1}$. This is typically done using a truncated SVD. The updated rank $r_{i+1}$ is chosen to meet prescribed error tolerances.

In theory, for a general initial guess and without rank restrictions, the DMRG method will eventually find the global optimum of the unconstrained Rayleigh quotient. However, for large systems, a maximum rank must be chosen to maintain computational feasibility. Once this maximum rank is reached, the DMRG suffers from the same issues as ALS: it might get stuck in local minima whose quality (in terms of the resulting energy) is unknown. 

In our experiments, we have observed another issue with the DMRG. Since we want to compare the output of the DMRG with critical points on the manifold $V_{\mathbf k, \mathbf r}^=$, we must truncate the ranks to the maximum rank $\mathbf r$ in each step. While the solution to the subproblem is optimal in terms of energy, the rank truncation of the solution is not. 
The truncated SVD yields the optimal approximation in the Frobenius norm but not in energy. Hence, the DMRG does not converge to critical points of the Rayleigh quotient on the manifold, but to other points with suboptimal energy. 
To address this comparison, we ran DMRG 1000 times for five random symmetric
matrices $D_1,\ldots,D_5$ and compared the resulting energies with the local
minima computed with \texttt{HomotopyContinuation.jl}.  The DMRG analog of Table~\ref{tab: ALS M} can be found under the link \cite{zenodo_our}. The results are summarized in
Table~\ref{tab:dmrg_summary}.
\begin{table}[ht]
\centering
\small
\begin{tabular}{|c|c|c|}
\toprule
Matrix & Global minimum & Best DMRG energy
       \\
\midrule
$D_1$ & $-2.832$  & $-2.8318$
 \\
$D_2$ & $-2.059$  & $-2.0574$
\\
$D_3$ & $-3.7147$ & $-3.5123$ \\
$D_4$ & $-1.9487$ & $-1.8445$ \\
$D_5$ & $-2.3341$ & $-2.2884$ \\
\bottomrule
\end{tabular}
\caption{DMRG energies for five random symmetric matrices
$D_1,\ldots,D_5$, compared with the global minimum 
computed with \texttt{HomotopyContinuation.jl}.}
\label{tab:dmrg_summary}
\end{table}

It is possible to replace the truncated SVD with a more advanced rank truncation that finds the best low-rank approximation in terms of the energy in each step (see \cite{Krumnow2021}), but this exceeds the scope of this article.

\begin{center}
 \textbf{Acknowledgments}
\end{center}
We are grateful to Bernd Sturmfels for introducing the authors. 
We thank Flavio Salizzoni, Luca Sodomaco, and Julian Weigert for many helpful conversations and for sharing their notes on RR degrees with us. 
We thank Otto Schmidt for interesting discussions and for providing several physical Hamiltonians for our experiments. Finally, we thank anonymous MEGA referees for their helpful comments, which helped to improve the presentation of the~paper.

VB acknowledges support
from the European Research Council \begin{footnotesize} (UNIVERSE PLUS, 101118787). \end{footnotesize}
$\!\!$ Views and opinions expressed
are however those of the authors only and do not necessarily reflect those of the European Union or the 
European
Research Council Executive Agency. Neither the European Union nor the granting authority
can be held responsible for them.

\begin{small}
\bibliographystyle{plain}
\bibliography{june18-arxiv}
\end{small}
\bigskip \medskip \bigskip

\noindent
\footnotesize {\bf Authors' addresses:}
\smallskip

\noindent Max Planck Institute for Mathematics in the Sciences Leipzig, Germany \hfill {\tt  borovik@mis.mpg.de}

\noindent University of California Berkeley, USA \hfill {\tt  hannahfriedman@berkeley.edu}

\noindent San Francisco State University, USA \hfill {\tt  serkan@sfsu.edu}

\noindent Georg-August-Universität Göttingen, Germany \hfill {\tt  
m.pfeffer@math.uni-goettingen.de}

\end{document}